\documentclass[12pt]{article}
\usepackage{amsmath, amsthm, amssymb, mathrsfs}
\usepackage{graphicx}
\usepackage[all,2cell]{xy}
\UseAllTwocells
\usepackage[all]{xy}
\usepackage{hyperref}
\usepackage{graphicx}
\usepackage{manfnt}

\def\bB{{\bf B}}

\def\bG{{\bf G}}
\def\bH{{\bf H}}
\def\bJ{{\bf J}}
\def\bK{{\bf K}}

\def\bM{{\bf M}}
\def\bN{{\bf N}}

\def\bS{{\bf S}}
\def\bT{{\bf T}}

\def\bZ{{\bf Z}}

\def\aq{/  \kern-.25em / }

\def\g{\mathfrak{g}}

\def\cH{\mathcal{ H}}

\def\Z{\mathbb{ Z}}

\def\boxit#1{\vbox{\hrule\hbox{\vrule\kern3pt
          \vbox{\kern3pt#1\kern3pt}\kern3pt\vrule}\hrule}}

\begin{document}

\newtheorem{theorem}{Theorem}[subsection]
\newtheorem{lemma}[theorem]{Lemma}
\newtheorem{proposition}[theorem]{Proposition}
\newtheorem{fact}[theorem]{Fact}
\newtheorem{corollary}[theorem]{Corollary}

\theoremstyle{definition}
\newtheorem{definition}[theorem]{Definition}
\newtheorem{example}[theorem]{Example}
\newtheorem{xca}[theorem]{Exercise}

\theoremstyle{remark}
\newtheorem{remark}[theorem]{Remark}

\def\goth{\frak}

\def\GL{{\rm GL}}
\def\tr{{\rm tr}\, }
\def\A{{\Bbb A}}
\def\bs{\backslash}
\def\Q{{\Bbb Q}}
\def\R{{\Bbb R}}
\def\Z{{\Bbb Z}}
\def\C{{\Bbb C}}
\def\SL{{\rm SL}}
\def\cS{{\cal S}}
\def\cH{{\cal H}}
\def\G{{\Bbb G}}
\def\F{{\Bbb F}}
\def\cF{{\cal F}}

\def\cB{{\cal B}}
\def\cA{{\cal A}}
\def\cE{{\cal E}}

\newcommand{\oB}{{\overline{B}}}
\newcommand{\oN}{{\overline{N}}}

\def\CC{{\Bbb C}}
\def\ZZ{{\Bbb Z}}
\def\QQ{{\Bbb Q}}
\def\cS{{\cal S}}

\def\Ad{{\rm Ad}}

\def\bG{{\bf G}}
\def\bH{{\bf H}}
\def\bT{{\bf T}}
\def\bM{{\bf M}}
\def\bB{{\bf B}}
\def\bN{{\bf N}}
\def\bS{{\bf S}}
\def\bZ{{\bf Z}}
\def\t{\kern.1em {}^t\kern-.1em}
\def\cc#1{C_c^\infty(#1)}
\def\Fx{F^\times}
\def\half{\hbox{${1\over 2}$}}
\def\T{{\Bbb T}}
\def\Ox{{\frak O}^\times}
\def\Ex{E^\times}
\def\Ecl{{\cal E}}

\def\g{{\frak g}}
\def\h{{\frak h}}
\def\k{{\frak k}}
\def\ft{{\frak t}}
\def\n{{\frak n}}
\def\b{{\frak b}}

\def\2by2#1#2#3#4{\hbox{$\bigl( 
{#1\atop #3}{#2\atop #4}\bigr)$}}
\def\wh{\Xi}
\def\C{{\Bbb C}}
\def\bs{\backslash}
\def\adots{\mathinner{\mkern2mu
\raise1pt\hbox{.}\mkern2mu
\raise4pt\hbox{.}\mkern2mu
\raise7pt\hbox{.}\mkern1mu}}

\def\tim{\bf}
\def\timit{\it}
\def\timsm{\scriptstyle}
\def\secttt#1#2{\vskip.2in\noindent
{$\underline{\hbox{#1}}$\break (#2)}\vskip.2in}
\def\sectt#1{\vskip.2in\noindent
{$\underline{\hbox{#1}}$}\vskip.2in}
\def\sectm#1{\vskip.2in\noindent
{$\underline{#1}$}\vskip.2in}
\def\mat#1{\left[\matrix{#1}\right]}
\def\cc#1{C_c^\infty(#1)}
\def\ds{\displaystyle}
\def\Hom{\mathop{Hom}\nolimits}
\def\Ind{\mathop{Ind}\nolimits}
\def\bs{\backslash}
\def\ni{\noindent}
\def\eb{{\bf e}}
\def\fb{{\bf f}}
\def\hb{{\bf h}}
\def\rg{{\goth R}}
\def\ig{{\goth I}}
\def\ccl{{\cal C}}
\def\dcl{{\cal D}}
\def\ecl{{\cal E}}
\def\hcl{{\cal H}}
\def\ocl{{\cal O}}
\def\ncl{{\cal N}}
\def\ag{{\goth a}}
\def\bg{{\goth b}}
\def\cg{{\goth c}}
\def\og{{\goth O}}
\def\hg{{\goth h}}
\def\lg{{\goth l}}
\def\mg{{\goth m}}
\def\Og{{\goth O}}
\def\rg{{\goth r}}
\def\sg{{\goth s}}
\def\Sg{{\goth S}}
\def\tg{{\goth t}}
\def\zg{{\goth z}}
\def\C{{\Bbb C}}
\def\Q{{\Bbb Q}}
\def\R{{\Bbb R}}
\def\T{{\Bbb T}}
\def\Z{{\Bbb Z}}
\def\t{{}^t\kern-.1em}
\def\tr{\hbox{tr}}
\def\ad{{\rm ad}}
\def\Ad{{\rm Ad}}
\def\2by2#1#2#3#4{\hbox{$\bigl( {#1\atop #3}{#2\atop #4}\bigr)$}}

\def\Kcl{{\cal K}}
\def\Pcl{{\cal P}}
\def\Scl{{\cal S}}
\def\Vcl{{\cal V}}
\def\Wcl{{\cal W}}
\def\Ecl{{\cal E}}
\def\A{{\Bbb A}}
\def\F{{\Bbb F}}
\def\T{{\Bbb T}}
\def\go{{\goth O}}
\def\Ox{{\goth O}^\times}
\def\Fx{F^\times}
\def\Ex{E^\times}
\def\half{\hbox{${1\over 2}$}}
\def\vtwo{\vskip .2in}
\def\BibFH{{\bf [1]}}
\def\BibGod{{\bf [2]}}
\def\BibCrelle{{\bf [3]}}
\def\BibAmJ{{\bf [4]}}
\def\BibDuke{{\bf [5]}}
\def\BibJL{{\bf [6]}}
\def\BibRR{{\bf [7]}}

\def\BibFH{{\bf [1]}}
\def\BibGod{{\bf [2]}}
\def\BibCrelle{{\bf [3]}}
\def\BibAmJ{{\bf [4]}}
\def\BibDuke{{\bf [5]}}
\def\BibJL{{\bf [6]}}
\def\BibRR{{\bf [7]}}

\newcommand{\N}{\mathbb{N}}

\newcommand{\gen}{{\operatorname{gen}}}
\newcommand{\ind}{\operatorname{ind}}
\newcommand{\Wh}{\mathcal{W}}
\newcommand{\Kr}{\mathcal{K}}
\newcommand{\V}{\mathcal{V}}
\newcommand{\U}{\mathcal{U}}
\renewcommand{\O}{\mathcal{O}}
\newcommand{\lift}{\goth{g}}
\newcommand{\inv}{\iota}
\newcommand{\supp}{{\operatorname{Supp}}}
\def\cW{{\cal W}}

\title{Distinguished Regular Supercuspidal Representations and\\ Inductive Constructions of Representations}
\author{Jeffrey Hakim}
\date{\today}
\maketitle
\abstract{This paper develops the theory of distinguished regular supercuspidal representations, and it highlights how the correspondence between regular characters and regular supercuspidal representations resembles induction in certain ways.
}
\tableofcontents

\parskip=.13in

\section{Introduction}

The purpose of this paper is to develop the theory of distinguished representations for Kaletha's regular supercuspidal representations (defined in \cite{KalYu}).

Our main result (Theorem \ref{mainthm}) is a formula
$$\langle \Theta , \pi (\mu)\rangle_G^1 =\sum_{\zeta\in T\bs \Theta_{\bT}} {\rm m}_T^G(\zeta)\ \langle \zeta ,\mu \rangle_T^\varepsilon,$$
with notations as follows:
\begin{itemize}
\item $(\bT,\mu)$ is a tame, elliptic, regular pair (in Kaletha's sense \cite[Definition 3.6.5]{KalYu}) relative to the connected, reductive group $\bG$ defined over a field $F$ that is a finite extension of $\Q_p$ (subject to the conditions on $p$ in \cite[\S2.1]{KalYu}).
\item $\pi (\mu)$ is the associated regular supercuspidal representation of $G= \bG(F)$ (as defined in \cite[\S3.8]{KalYu}).
\item $\Theta$ is a $G$-orbit of involutions of $G$, where ``involution of $G$'' means ``$F$-automorphism of order two'' and where $G$ acts on involutions by $(g\cdot \theta )(g') = g\theta (g^{-1}g'g)g^{-1}$.
\item $\langle \Theta , \pi (\mu)\rangle_G^1$ is the dimension of the space ${\rm Hom}_{G^\theta}(\pi (\mu),1)$, for some (hence all) $\theta\in \Theta$, where $G^\theta$ denotes the group of fixed points of $\theta$ in $G$.
\item $T\bs \Theta_\bT$ is the set of $T$-orbits in the set
$$\Theta_\bT = \{ \theta\in \Theta\ : \  \theta (\bT) = \bT\},$$ where $T=\bT (F)$.
\item ${\rm m}_T^G (\zeta) = [G_{\theta} :T_{\theta} G^{\theta}]$, for any element $\theta$ in $\zeta$, where $$G_\theta = \{ g\in G\ : g\cdot \theta = \theta\},$$ and $T_\theta = T\cap G_\theta$.
\item $\varepsilon = (\varepsilon_\theta)_{\theta \in \Theta_\bT}$ is a family of quadratic characters $\varepsilon_\theta : T^\theta \to \{ \pm 1\}$ defined in \S\ref{sec:epsi}.
\item $\langle \zeta ,\mu \rangle_T^\varepsilon =\begin{cases}1,&\text{if $\mu | T^\theta = \varepsilon_\theta$ for one (hence all) $\theta\in \zeta$},\\
0,&\text{otherwise}.\end{cases}$
\end{itemize}

One may view the correspondence $\mu \mapsto \pi (\mu)$ as being ``induction-like,'' in the same spirit that cohomological induction (e.g., Deligne-Lusztig induction) resembles ordinary induction of representations of finite groups in certain ways.
To justify this point of view, we offer the following non-standard statement of  a standard formula (due to Mackey) that pertains to ordinary induction:
$$\langle \Theta , \pi (\mu)\rangle_G^1 =\sum_{\zeta\in T\bs \Theta} {\rm m}_T^G(\zeta)\ \langle \zeta ,\mu \rangle_T^1,$$
where:
\begin{itemize}
\item $\mu$ is a representation of a subgroup $T$ of a finite group $G$.
\item $\pi (\mu)$ is the induced representation ${\rm Ind}_T^G (\mu)$ of $G$.
\item $\Theta$ is a $G$-orbit of involutions of $G$, where ``involution of $G$'' means ``automorphism of order two'' and where $G$ acts on involutions by $(g\cdot \theta )(g') = g\theta (g^{-1}g'g)g^{-1}$.
\item $\langle \Theta , \pi (\mu)\rangle_G^1$ is the dimension of the space ${\rm Hom}_{G^\theta}(\pi (\mu),1)$, for some (hence all) $\theta\in \Theta$, and $G^\theta$ denotes the group of fixed points of $\theta$ in $G$.
\item $T\bs \Theta$ is the set of $T$-orbits in $\Theta$.
\item ${\rm m}_T^G (\zeta) = [G_{\theta} :T_{\theta} G^{\theta}]$, where $\theta$ is any element of $\zeta$, and $G_\theta = \{ g\in G\ : g\cdot \theta = \theta\}$, and $T_\theta = T\cap G_\theta$.
\item $\langle \zeta ,\mu \rangle_T^1$  
is the dimension of the space ${\rm Hom}_{T^\theta}( \mu,1)$, for some (hence all) $\theta\in \zeta$, where $T^\theta = T\cap G^\theta$.
\end{itemize}

To see the latter formula in its familiar form, and for further details, we refer to \S\ref{sec:standardMackey}.  As a matter of convenience, we have stated the formula in the setting of finite groups, but it is routine to generalize it to other settings in which the standard Mackey formula holds.

A similar formula for Deligne-Lusztig induction is given in \cite{ANewHM} and it is recalled below in \S\ref{sec:Lusztigformula}.

These examples hint at a general framework for induction and induction-like constructions.  We partially develop such a framework in this paper and use it in the proof of our main result.

Just as one can iteratively induce a representation up a tower of subgroups, one can also imagine iterating a sequence of generalized induction operations.  Our approach to constructing regular supercuspidal representations involves this type of iteration.  In this regard, it follows \cite{ANewTasho}, rather than \cite{KalYu}.

Let us now explain  this in more detail. 
In \cite{ANewYu}, we have revised Jiu-Kang Yu's construction \cite{MR1824988} of supercuspidal representations, and in \cite{ANewTasho} we have applied the results of \cite{ANewYu} to give a more direct construction of Kaletha's map $\mu\mapsto \pi(\mu)$.  A major objective of this paper is to demonstrate that this direct construction of Kaletha's regular supercuspidal representations can be useful in simplifying the development of applications.  Specifically, this is shown in the context of the theory of distinguished representations.

Though the construction of supercuspidal representations has simplified in various ways, it remains rather technical and requires numerous notations.  This, unfortunately, limits our ability to write a paper that is both self-contained and brief.  So, in the interest of brevity, we refer to  \cite{ANewYu} and \cite{ANewTasho} for background material and many of our notations.

In our approach, the first step in the construction of a regular supercuspidal representation $\pi(\mu)$ of $G$  is the construction of a  representation $\rho_T^\mu$ of the  subgroup $TH_{x,0}$ of $G$, as in \S2.4--2.6 of  \cite{ANewTasho}.  Here, $H_{x,0}$ is a certain maximal parahoric subgroup of a reductive subgroup $H$ of $G$ (as in \cite{ANewTasho}). 
When $\pi (\mu)$ has depth zero, the representation $\rho_T^\mu$  is identical to Kaletha's representation $\tilde\kappa_{(S,\theta)}$.  (See \cite[\S3.4.4]{KalYu}.) In general, $\rho_T^\mu$ is  obtained, roughly speaking, by inflating a Deligne-Lusztig representation to  $H_{x,0}$ and then twisting by a character of $T$.

As with Deligne-Lusztig induction, the correspondence $\mu\mapsto \rho_T^\mu$ is induction-like, as it admits a formula
$$\left\langle \xi ,\rho_T^\mu\right\rangle^{\varepsilon^+}_{TH_{x,0}} 
= 
\sum_{\zeta\in T\bs \xi_{\bT}} {\rm m}_T^{TH_{x,0}} (\zeta) 
\left\langle \zeta , \mu\right\rangle^{\varepsilon}_{T},$$
similar to the formulas above.  (See Proposition \ref{maincoeff}.)

The next step in the construction of $\pi (\mu)$ is to induce $\rho_T^\mu$ to a representation $\rho$ of the subgroup $H_x$.  (See \S2.2 \cite{ANewTasho} and the introduction of \cite{ANewTasho} for the definition of $H_x$.)  There is an accompanying formula
$$\langle \vartheta , \rho\rangle_{H_x}^{\varepsilon^+} = \sum_{\xi \in (TH_{x,0})\bs \vartheta} {\rm m}_{TH_{x,0}}^{H_x}(\xi)\ \langle \xi ,\rho_T^\mu\rangle^{\varepsilon^+}_{TH_{x,0}} .$$  (See \S\ref{sec:standardMackey}.)

The latter two formulas can be combined to give
$$\langle \vartheta , \rho\rangle_{H_x}^{\varepsilon^+} = \sum_{\zeta\in T\bs \vartheta_{\bT}} {\rm m}_T^{H_{x}} (\zeta) 
\left\langle \zeta , \mu\right\rangle^{\varepsilon}_{T} .$$
Here, we are using the transitivity property (proven in \S\ref{sec:indweights})
$$ {\rm m}_{TH_{x,0}}^{H_x}(\xi) \cdot {\rm m}_T^{TH_{x,0}} (\zeta)  =  {\rm m}_T^{H_{x}} (\zeta)$$
that applies when $\zeta \subset \xi$.  (In our induction framework, it is this transitivity property that is the key to the transitivity of induction.)

The representation $\rho$ is an example of a ``permissible representation,'' in the terminology of \cite{ANewYu}.  The main focus  of \cite{ANewYu} is the development of an induction-like modification of  Yu's construction of tame supercuspidal representations so that it maps permissible representations $\rho$ to supercuspidal representations $\pi (\rho)$ of $G$.  In the present case, the representation $\pi (\rho)$, with $\rho$ associated to $\mu$ as sketched above, is the regular supercuspidal representation $\pi (\mu)$ associated to $\mu$.

The main result of \cite{ANewHM} is the formula
$$\langle \Theta,\pi(\rho)\rangle_G^1 = \sum_{\vartheta\in H_x\bs \Theta_{H_x}} {\rm m}^G_{H_x}(\vartheta)\, \langle \vartheta , \rho\rangle_{H_x}^{\varepsilon^+},$$
where $\Theta_{H_x}$ is the set of $\theta\in \Theta$ such that $\theta (H_x) = H_x$.
Because of this formula, we regard the construction $\rho\mapsto \pi(\rho)$ as being induction-like.
Combining this formula with the previous formulas gives the main result in this paper (stated above).

Above we have referred to a ``framework'' that encompasses some well-known variations on representation-theoretic induction.  But, in truth, we have only presented the beginnings of a framework. 
In particular, we do not formally define what it means for a representation-theoretic construction to be induction-like.  Informally, if $H$ is a subgroup of a group $G$ then a correspondence 
$\mu\mapsto \pi (\mu)$ from representations of $H$ to representations of $G$ is {\it induction-like} if it satisfies formulas
$$\langle \Theta , \pi (\mu)\rangle_G^{\chi} =\sum_{\zeta} {\rm m}_H^G(\zeta)\ \langle \zeta ,\mu \rangle_H^{\chi\varepsilon}$$
analogous to those presented above, where $\zeta$ is summed over a suitable subset of $H\bs \Theta$.
 It is not yet clear whether the preliminary steps taken in this paper  will lead to a true general framework.
 (See Definition \ref{pairingdef} for a precise definition of the pairings.)

One can use our main formula to determine the distinguished regular supercuspidal representations of $G$.  Doing so yields the main theorem in Chong Zhang's paper \cite{CZdist}, which says that $\pi (\mu)$ is distinguished with respect to some (hence every) involution in $\Theta$ precisely when there exists an involution $\theta\in \Theta$ such that $\theta (\bT) = \bT$ and $\mu|T^\theta = \varepsilon_\theta$.

The main formula in this paper resembles Proposition 8.11 \cite{MR3027804} which applies to the special case in which $G = \GL_n (F)$ and $G^\theta$ is an orthogonal group in $n$-variables.  That result is used in \cite{MR3027804} to prove a symmetric space result that the space
$$ {\rm Hom}_G( \pi ,C^\infty ((\bG^\theta \bs \bG)(F)))$$  has dimension 0 or 4
for every irreducible, tame supercuspidal representation $\pi$ of $G$.  (See Theorem 1.1 \cite{MR3027804} for a more precise statement.)
This Hom-space is naturally expressed as a direct sum of Hom-spaces param\-e\-trized by pure inner forms of $\bG^\theta$.
It is only natural to suspect that more such examples, and perhaps a general symmetric space result, will follow from the results in this paper.
We hope to consider this issue in a sequel to this paper.

Finally, we wish to gratefully acknowledge that this paper is greatly influenced by the our earlier collaborations with Fiona Murnaghan and Joshua Lansky on the theory of distinguished supercuspidal representations, most notably the papers \cite{MR2431732} and \cite{MR2925798}.

\section{Induction-Like Constructions}

\subsection{Mackey's formula}\label{sec:standardMackey}

One of the most basic formulas in representation theory is Mackey's formula
$${\rm Hom}_K({\rm Ind}_H^G(\rho),\chi)\cong \bigoplus_{HgK\in H\bs G/K}{\rm Hom}_{H\cap gKg^{-1}}(\rho,{}^g\chi),$$
where:
\begin{itemize}
\item $H$ and $K$ are subgroups of a finite group $G$,
\item $\chi$ is a character of $K$,
\item $\rho$ is a complex representation of $H$,
\item ${}^g\chi (h) = \chi(g^{-1}hg)$.
\end{itemize}
The isomorphism is given explicitly as $$\lambda\leftrightarrow (\lambda_g),$$
where, given $v$ in the space of $\rho$ 
$$\lambda_g(v) = \lambda(f_{g,v}),$$ with $$f_{g,v}(hgk) = \rho(hgkg^{-1})v,$$ for $h\in h$ and  $k\in K$ and $f_{g,v}\equiv 0$ outside $HgK$.
So $\lambda_g$ controls the restriction of $\lambda$ to the functions in ${\rm Ind}_H^G(\rho)$ with support in $HgK$.

If we replace $(G,H,K)$, with $(H,K,H^\theta)$,
where   $H^\theta$ is the group of fixed points of an automorphism $\theta$ of $G$ of order two, then the Mackey formula says
$${\rm Hom}_{H^\theta} ({\rm Ind}_K^H(\rho),\chi)\cong \bigoplus_{KhH^\theta \in K\bs H/H^\theta }{\rm Hom}_{K\cap hH^\theta h^{-1}}(\rho,{}^h\chi).$$
We observe that $hH^\theta h^{-1}$ is the group $H^{h\cdot \theta}$ of fixed points of the automorphism
$(h\cdot \theta)(g) = h\theta (h^{-1}gh)h^{-1}$.
Now let $H_\theta$ be the group of elements $h\in H$ such that $h\cdot \theta = \theta$ and assume that $\chi$ has the property that ${}^h\chi = \chi$ whenever $h\in H_\theta$.

Now suppose $h_1$ and $h_2$ are elements of $H$ that are in different double cosets in $K\bs H/H^\theta$, but $h_1H_\theta = h_2H_\theta$ (and hence $h_1\cdot \theta = h_2\cdot \theta$).  Then the summands associated to $h_1$ and $h_2$ are identical, that is,
$${\rm Hom}_{K\cap H^{h_1\cdot\theta }}(\rho,{}^{h_1}\chi) =
{\rm Hom}_{K\cap H^{h_2\cdot \theta}}(\rho,{}^{h_2}\chi).$$
If we aggregate such repeated terms, the Mackey formula becomes
$${\rm Hom}_{H^\theta} ({\rm Ind}_K^H(\rho),\chi)\cong \!\!\!\!\!\bigoplus_{KhH^\theta \in K\bs H/H_\theta }\!\!\!\!\! {\rm m}_K^H(h\cdot \theta)\cdot {\rm Hom}_{K\cap H^{h\cdot \theta}}(\rho,{}^h\chi),$$
for  multiplicity constants $${\rm m}_K^H(\theta') = [H_{\theta'} : (K\cap H_{\theta'} )H^{\theta'}].$$

Now suppose $\vartheta$ is an $H$-orbit of automorphisms of $H$ of order two.  Let $\chi = (\chi_\theta)_{\theta\in \vartheta}$ be a family of characters, where $\chi_\theta$ is a character of $H^\theta$ and $$\chi_{h\cdot \theta}(hkh^{-1}) = \chi_\theta(k),$$ for all $h\in H$ and $k\in H^\theta$.

Let
$$\langle \vartheta ,{\rm Ind}_K^H(\rho)\rangle_H^\chi = \dim {\rm Hom}_{H^\theta} ({\rm Ind}_K^H(\rho),\chi_\theta),$$ for some (hence all) $\theta\in \vartheta$.
Similarly, when $\xi$ is a $K$-orbit in $\vartheta$ let
$$\langle \xi , \rho\rangle_K^\chi = \dim {\rm Hom}_{H^\theta} ( \rho,\chi_\theta),$$
for some (hence all) $\theta\in \xi$.

The Mackey formula now implies:

\begin{proposition}
Let $K\subset H$ be subgroups of a finite group $G$, and let $\rho$ be an irreducible, complex representation of $K$.  Let $\vartheta$ be an $H$-orbit of automorphisms of $G$ of order two.   Let $\chi = (\chi_\theta)_{\theta\in \vartheta}$ be a family of characters $\chi_\theta : H^\theta\to \C^\times$ such that $\chi_{h\cdot \theta}(hkh^{-1}) = \chi_\theta(k)$, for all $h\in H$ and $k\in H^\theta$.  Then
$$\left\langle \vartheta ,{\rm Ind}_K^H(\rho)\right\rangle_H^\chi = \sum_{\xi\subset\vartheta} {\rm m}_K^H(\xi) \ 
\langle \xi , \rho\rangle_K^\chi,$$ where the sum is taken over all $K$-orbits $\xi$ in $\vartheta$.
\end{proposition}

\subsection{Lusztig's formula}\label{sec:Lusztigformula}

Given a general position character $\lambda$ of the $\F_q$-rational points $\mathsf{T}(\F_q)$ of an elliptic maximal $\F_q$-torus $\mathsf{T}$ in a connected, reductive $\F_q$-group $\mathsf{G}$,  
Deligne and Lusztig \cite{MR0393266} have defined an  irreducible cuspidal representation $$
{}^{\scriptscriptstyle\mathsf{DL}}
\mathsf{ind}_{\mathsf{T}}^{\mathsf{G}} (\lambda)= \pm R_{\mathsf{T}}^\lambda$$ of $\mathsf{G}(\F_q)$.  The Deligne-Lusztig map $$\lambda \mapsto {}^{\scriptscriptstyle\mathsf{DL}}
\mathsf{ind}_{\mathsf{T}}^{\mathsf{G}} (\lambda)$$ has functorial properties similar to the properties of representation-theoretic induction.   Accordingly, it is called ``Deligne-Lusztig induction.''

Both the similarities and differences between Deligne-Lusztig induction and standard induction are evident if one tries to obtain a Mackey-type formula for Deligne-Lusztig induction.

Suppose $q$ is odd.   Let  $\Theta$ is a $\mathsf{G}(\F_q)$-orbit of involutions of $\mathsf{G}(\F_q)$, where ``involution of $\mathsf{G}(\F_q)$'' means ``$\F_q$-automorphism of $\mathsf{G}$ of order two,'' and the action of $\mathsf{G}(\F_q)$ on involutions is again defined according to $(g\cdot\theta) (h) = g\theta(g^{-1}hg)g^{-1}$.

Let $\mathsf{T}(\F_q)\bs \Theta_{\mathsf{T}(\F_q)}$ be the set of $\mathsf{T}(\F_q)$-orbits in the set
$$\Theta_{\mathsf{T}(\F_q)}= \left\{ \theta\in \Theta\ :\ \theta (\mathsf{T})= \mathsf{T}\right\}.$$
Given $\theta\in  \Theta_{\mathsf{T}(\F_q)}$, define a character
$$\varepsilon_\theta :\mathsf{T}(\F_q)^\theta\to \{ \pm \}$$ by
$$\varepsilon_\theta (t) = \det \left( \Ad (t)\, \big|\,  {\rm Lie}(\mathsf{G})^\theta (\F_q)) \right).$$
(See \cite[\S4.3]{ANewHM} for other expressions of this character.)

Let $$\left\langle \Theta ,{}^{\scriptscriptstyle\mathsf{DL}}
\mathsf{ind}_{\mathsf{T}(\F_q)}^{\mathsf{G}(\F_q)} (\lambda)\right\rangle_{\mathsf{G}(\F_q)}^1 = \dim {\rm Hom}_{\mathsf{G}(\F_q)^\theta} \left({}^{\scriptscriptstyle\mathsf{DL}}
\mathsf{ind}_{\mathsf{T}(\F_q)}^{\mathsf{G}(\F_q)} (\lambda), 1\right),$$ for some (hence all) $\theta\in \Theta$.

Suppose $\xi\in \mathsf{T}(\F_q)\bs \Theta_{\mathsf{T}(\F_q)}$.  Let
$${\rm m}_{\mathsf{T}(\F_q)}^{\mathsf{G}(\F_q)}(\xi) =  [\mathsf{G}(\F_q)_{\theta} : (\mathsf{T}(\F_q)\cap \mathsf{G}(\F_q)_{\theta} )\mathsf{G}(\F_q)^{\theta}],$$ for some (hence all) $\theta\in \xi$.  Let 
$$\langle \xi , \lambda\rangle_{\mathsf{T}(\F_q)}^\varepsilon = \begin{cases}1,&\text{if $\lambda | \mathsf{T}(\F_q)^\theta = \varepsilon_\theta$ for one (hence all) $\theta\in \xi$},\\
0,&\text{otherwise}.\end{cases}
$$

The following result is a special case of Theorem 2.0.1 \cite{ANewHM}:

\begin{proposition}
$$\left\langle \Theta ,{}^{\scriptscriptstyle\mathsf{DL}}
\mathsf{ind}_{\mathsf{T}(\F_q)}^{\mathsf{G}(\F_q)} (\lambda)\right\rangle_{\mathsf{G}(\F_q)}^1 = \sum_{\xi\in \mathsf{T}(\F_q)\bs \Theta_{\mathsf{T}(\F_q)}} {\rm m}_{\mathsf{T}(\F_q)}^{\mathsf{G}(\F_q)}(\xi) \ 
\langle \xi , \lambda\rangle_{\mathsf{T}(\F_q)}^\varepsilon$$
\end{proposition}

The relation between this result and the main results of \cite{MR1106911} is explained in \cite{ANewHM}.  (See also \cite{MR2925798}.)

\subsection{The induction weights ${\rm m}_K^H (\theta)$}\label{sec:indweights}

The material in this section can be developed in various representation-theoretic settings.  We focus on the setting that fits the main application of this paper.
In general, one needs to have a  collection (or category) of acceptable groups and representations, as well a notion  of ``involution'' for a given group.

For our purposes, it is  convenient to have a  universal group that contains all other groups under consideration.  For us, that group is the group $G= \bG (F)$, where $\bG$ is a connected, reductive group over $F$.  All other groups in this section will be subgroups of $G$ (or $\bG$).

When we use the terminology ``involution'' or ``involution of $G$,'' we mean an $F$-auto\-mor\-phism of $\bG$ of order two.   In general, if $\bH$ is an $F$-subgroup of $\bG$ then we take
$$\bH^\theta = \{ h\in \bH\ :\ \theta (h) = h\}$$
and 
\begin{eqnarray*}
\bH_\theta&=& \{ h\in \bH\ :\ h\cdot \theta =\theta\}\\
&=& \{h\in \bH\ :\ h\theta(h)^{-1}\in Z(\bG)\},\end{eqnarray*}
where $Z(\bG)$ denotes the center of $\bG$.

The group $\bH_\theta$ has other natural descriptions that are given in \cite[Corollary 1.3]{MR1215304}.
Note that we are {\it not} defining $\bH_\theta$ to be the stabilizer of $\theta |\bH$.  We also are not assuming that $\bH$ is $\theta$-stable.

A motivating example, considered in \cite{MR1674664, MR2925798, MR3027804}, is the example in which $\bH = \GL_n$, $\bH^\theta$ is an orthogonal group, and $\bH_\theta$ is the corresponding orthogonal similitude group.  
In general, we   view the homomorphism
$$\mu :\bH_\theta / \bH^\theta\to Z(\bG) : h\bH^\theta \mapsto h\theta(h)^{-1}$$
as an analogue of  the similitude homomorphism from the motivating example.  This homomorphism
realizes $\bH_\theta / \bH^\theta$ as a subgroup of $Z(\bG)$.  In particular, it is abelian.

For $F$-subgroups $\bH$ of $\bG$, we use the notation $H = \bH(F)$.  In other words, taking $F$-rational points is reflected by non-boldface notations.

\begin{definition}
If $\bK\subset\bH$ are $F$-subgroups of $\bG$ and $\theta$ is an involution of $G$ then define
$${\rm m}_K^H(\theta) = [H_\theta : K_\theta H^\theta].$$
When $S$ is a set of involutions of $G$ such that ${\rm m}_K^H(\theta)$ is constant as $\theta$ varies over $S$, let ${\rm m}_K^H(S)$ denote this common value of ${\rm m}_K^H(\theta)$.
\end{definition}

Let $G$ act on the set of involutions of $G$ by
$$g\cdot \theta = {\rm Int}(g)\circ\theta\circ {\rm Int}(g)^{-1},$$ where ${\rm Int}(g)(h) = ghg^{-1}$.

\begin{lemma}
If $\bK\subset\bH$ are $F$-subgroups of $\bG$, and $\theta$ is an involution of $G$, and $g\in G$, then 
$${\rm m}_K^H(g\cdot \theta) = {\rm m}_{g^{-1}Kg}^{g^{-1}Hg}( \theta).$$
\end{lemma}

\begin{proof}
\begin{eqnarray*}
{\rm m}_K^H(g\cdot \theta)
&=& [H_{g\cdot \theta} : K_{g\cdot\theta} H^{g\cdot \theta}]\\
&=& [g((g^{-1}Hg)_\theta )g^{-1} : g((g^{-1}Kg)_\theta (g^{-1}Hg)^\theta )g^{-1}]\\
&=& [(g^{-1}Hg)_{ \theta} : (g^{-1}Kg)_{\theta} (g^{-1}Hg)^{\theta}]\\
&=&{\rm m}_{g^{-1}Kg}^{g^{-1}Hg}( \theta).
\end{eqnarray*}
\end{proof}

\begin{corollary}
If $\bK\subset\bH$ are $F$-subgroups of $\bG$, and if 
$\mathscr{O}$ is a $K$-orbit of involutions of $G$ then 
${\rm m}_K^H(\theta)$ is constant as $\theta$ varies over $\mathscr{O}$ and thus the notation ${\rm m}_K^H(\mathscr{O})$ is well-defined.
\end{corollary}

The next result is the basic transitivity property of the ${\rm m}_K^H(\theta)$'s  that is linked with the transitivity of induction.

\begin{lemma}\label{transofind}
If $\bJ\subset\bK\subset\bH$ are $F$-subgroups of $\bG$ and $\theta$ is an involution of $g$ then ${\rm m}_J^H(\theta)= {\rm m}_J^K( \theta) {\rm m}_K^H( \theta)$.
\end{lemma}

\begin{proof}
\begin{eqnarray*}
{\rm m}_J^H(\theta )
&=&[H_\theta :J_\theta H^\theta]\\
&=&[H_\theta :K_\theta H^\theta][K_\theta H^\theta :J_\theta H^\theta]\\
&=&[H_\theta :K_\theta H^\theta][K_\theta  :J_\theta (K_\theta \cap H^\theta )]\\
&=&[H_\theta :K_\theta H^\theta][K_\theta  :J_\theta K^\theta ]\\
&=& {\rm m}_J^K( \theta) {\rm m}_K^H( \theta).
\end{eqnarray*}
\end{proof}

Suppose $\bK\subset\bH$ are $F$-subgroups of $\bG$.  
Given an involution $\theta$ of $G$, we can consider the set $\mathscr{O}_K(H\cdot \theta)$ of $K$-suborbits of the $H$-orbit $H\cdot \theta$ of $\theta$

Define a map
$$K\bs H/H^\theta \to \mathscr{O}_K(H\cdot \theta) : KhH^\theta \mapsto Kh\cdot \theta.$$
Then the fiber of the $K$-orbit $Kh\cdot \theta$ is the set 
$$\mathscr{F}_h^K = \{ Kh'H^\theta \in K\bs H/H^\theta \ :\ Kh'\cdot \theta =Kh\cdot \theta\}.$$

\begin{lemma}
If $h\in H$ then
$\mathscr{F}_h^K = h\mathscr{F}^{h^{-1}Kh}_1$ and the cardinality of $\mathscr{F}_h^K$  is $ {\rm m}_K^H(h\cdot \theta)$.
\end{lemma}

\begin{proof}The identity $\mathscr{F}_h^K = h\mathscr{F}^{h^{-1}Kh}_1$ follows from the definitions:
\begin{eqnarray*}
h\mathscr{F}^{h^{-1}Kh}_1
&=&\{ Khh'H^\theta\ :\ h^{-1}Khh'\cdot h^{-1}Kh\cdot \theta\}\\
&=&\{ Kh'H^\theta \in K\bs H/H^\theta \ :\ Kh'\cdot \theta =Kh\cdot \theta\}\\
&=&\mathscr{F}^{K}_h
\end{eqnarray*}

To prove the second assertion, it suffices to handle the special case in which $h=1$.  Indeed, if the $h=1$ case holds then 
\begin{eqnarray*}
|\mathscr{F}_h^K| &=& | h\mathscr{F}_1^{h^{-1}Kh}|
= |\mathscr{F}_1^{h^{-1}Kh} |
= {\rm m}_{h^{-1}Kh}^H(\theta)
= {\rm m}_{h^{-1}Kh}^{h^{-1}Hh}(\theta)\\
&=& {\rm m}_{K}^H(h\cdot \theta).
\end{eqnarray*}

Suppose $h'\in H$.  Then $Kh'H^\theta\in \mathscr{F}_1^H$ precisely when there exists $k\in K$ such that $kh'\in H_\theta$, but this occurs precisely when $Kh'H^\theta$ contains an element of $H_\theta$.  But the double cosets with a representative in $H_\theta$ are in bijection with $K_\theta\bs H_\theta/H^\theta$.  Since $H_\theta/H^\theta$ is a subgroup of the abelian group $\bG_\theta/\bG^\theta$, the cardinality of $K_\theta\bs H_\theta/H^\theta$ is the index of the subgroup $(K_\theta H^\theta)/H^\theta$ of $H_\theta/H^\theta$.  But this is just ${\rm m}_K^H(\theta)$.  So $|\mathscr{F}_1^K| = {\rm m}_K^H(\theta)$, which completes the proof.
\end{proof}

\subsection{Compatible families of characters and the pairings $\langle \mathscr{O},\tau\rangle_K^\chi$}\label{sec:compatsec}

Let us maintain the notations and setup of the previous section and, in particular, assume that $\bH$ is an $F$-subgroup of $\bG$.  In addition, we fix an $H$-orbit $\mathscr{O}$ of involutions of $G$ and a (smooth, complex) representation $\tau$ of $H$.

\begin{definition}\label{compatico}
A family $\chi = (\chi_\theta)_{\theta\in \mathscr{O}}$ of characters $\chi_\theta :H^\theta \to \C^\times$ is called a {\bf compatible family} if $$\chi_{h\cdot \theta} (hkh^{-1}) = \chi_\theta (k),$$ whenever $\theta\in \mathscr{O}$, $h\in H$ and $k\in H^\theta$.
\end{definition}

Observe that, in the previous definition, a given character $\chi_\theta$ is forced to have trivial restriction to
$$[H_\theta ,H^\theta] = \{ hkh^{-1}k^{-1}\ :\ h\in H_\theta,\, k\in H^\theta\}.$$

We also remark that if $\chi$ and $\chi'$ are two compatible families then the pointwise product $\chi\chi'$ is another compatible family.

\begin{definition}\label{pairingdef}  If $\chi$ is a compatible family, define
$\langle \mathscr{O},\tau\rangle^\chi_H$ to be the dimension of ${\rm Hom}_{H^\theta}(\tau,\chi_\theta )$ for any (hence all) $\theta\in \mathscr{O}$.
\end{definition}

The fact that $\langle \mathscr{O},\tau\rangle^\chi_H$ is indeed well-defined can be explained as follows.  Suppose $\lambda_\theta$ lies in ${\rm Hom}_{H^\theta}(\tau,\chi_\theta )$.  Then $\lambda_\theta :V_\tau\to \C$, where $V_\tau$ is the space of $\tau$.  Given $h\in H$, we define another linear form $\lambda_{h\cdot \theta}$ by
$$\lambda_{h\cdot \theta} (v) = \lambda_\theta (\tau (h)^{-1}v),$$ for $v\in V_\tau$.  Then it is routine to verify that $\lambda_\theta \mapsto \lambda_{h\cdot \theta}$ defines a bijection
$${\rm Hom}_{H^\theta}(\tau,\chi_\theta )\cong {\rm Hom}_{H^{h\cdot \theta}}(\tau,\chi_{h\cdot\theta} ).$$  This implies $\langle \mathscr{O},\tau\rangle^\chi_H$ is well-defined.

\section{Distinguished regular supercuspidal representations}

\subsection{Background}\label{sec:tsr}

This paper is an application of the construction in \cite{ANewTasho} of regular supercuspidal representations.  It differs from Kaletha's construction \cite{KalYu} in that it uses the revision of Yu's construction \cite{MR1824988} that appears in \cite{ANewYu}.  We also draw on \cite{MR2431732} and \cite{MR2925798} for basic facts in the theory of distinguished tame supercuspidal representations.  With regards to notations and terminology, we generally follow 
\cite{ANewYu}, \cite{ANewHM}, \cite{ANewTasho}.

We now sketch the construction of regular supercuspidal representations providing references as needed.
Fix a finite extension $F$ of $\Q_p$ for $p\ne 2$.  (The fact that we assume characteristic zero is a matter of convenience, and we have no reason to suspect significant problems for positive characteristics.)  Next, we fix a connected reductive $F$-group $\bG$ and let $G = \bG (F)$.  (This use of boldface for $F$-objects, and non-boldface for $F$-points is used throughout this paper.)  We assume $p$ is not a bad prime for $\bG$.
Recall that, excluding the prime 2, the only bad primes are 3 (for all exceptional types) and 5 (for type $E_8$).

Fix an algebraic closure $\overline{F}$ of $F$ and let $F^{\rm un}$ be the maximal unramified extension of $F$ in $\overline{F}$.  The residue field $\mathfrak{F}$ of $F^{\rm un}$ is an algebraic closure for the residue field $\mathfrak{f}$ of $F$.  

We also assume that $p$ does not divide the order of the fundamental group of the derived group $\bG_{\rm der}$ of $\bG$.  As explained in Sections 2.1 and 2.3 of \cite{ANewTasho}, this assumption can be circumvented by passing to a $z$-extension of $\bG$.
Our assumptions on $p$ are identical to those in \cite{KalYu}.

All representations we consider should be assumed to be complex, smooth representations, and we use the term ``character'' for a 1-dimensional representation.

Suppose $\bT$ is a tame, elliptic, maximal $F$-torus in $\bG$ and suppose $\mu$ is a character of $T = \bT (F)$.  As in \S2.2 \cite{ANewTasho}, we associate to $\mu$ a tower
$$\bH = \bG^0 \subsetneq \cdots \subsetneq \bG^d = \bG$$ of  subgroups of $\bG$, a point
$$x\in \mathscr{B}_{\rm red}(\bG,F),$$ and a sequence of positive numbers
$$r_0<r_1<\cdots <r_d.$$
Let
$$\vec\bG =(  \bG^0,\dots , \bG^d)$$
and $$\vec r = (r_0,\dots , r_d).$$

Definition 3.6.5 in \cite{KalYu} is equivalent to the following (which is Definition 2.2.2 \cite{ANewTasho}):

\begin{definition}\label{terp}
A pair $(\bT,\mu)$ is  {\bf a tame, elliptic, regular pair} if:
\begin{itemize}
\item[(1)] $\bT$ is a tame,  elliptic, maximal $F$-torus in $\bG$.
\item[(2)] $\mu$ is a character of $T$ such that $\bT$ is a maximally unramified subtorus of $\bH$.
\item[(3)] Any element of $H$ that normalizes $\bT$ and fixes $\mu |T_{0}$ must lie in $T$.
\end{itemize}
\end{definition}

Maximally unramified tori are discussed in \cite[\S3.4.1]{KalYu}.

In \cite{ANewTasho}, in place of Kaletha's Howe factorizations \cite[\S3.7]{KalYu}, we use the following:

\begin{definition}\label{weakfact}
If $(\bT,\mu)$ is a tame, elliptic, regular pair for $\bG$ then a {\bf weak  factorization} of $\mu$ is a pair $(\mu_-,\mu_+)$ consisting of a character $\mu_-$ of $T$ and a character $\mu_+$ of $H_\mu$ such that
\begin{itemize}
\item[$\bullet$] $(\bT, \mu_-)$ is a depth zero tame, elliptic, regular pair for  $\bH_\mu$.
\item[$\bullet$]  $\mu = \mu_- (\mu_+|T)$.
\end{itemize}
\end{definition}

\noindent We prove in \cite[Lemma 2.3.2]{ANewTasho} that weak factorizations always exist.  We also observe that Lemma 3.4.2 \cite{KalYu} implies that  $x$ must be a vertex in $\mathscr{B}_{\rm red}(\bH,F)$.

In \cite[\S2.6]{ANewTasho}, we use the weak factorization to construct a certain representation $\rho_T^\mu$ of the compact-mod-center subgroup $TH_{x,0}$.
Then we take $\rho$ to be the representation of $H_x$ induced by $\rho_T^\mu$.  Here, $H_x$ is the stabilizer of $x$ in $H$ or, equivalently, the normalizer of the parahoric subgroup $H_{x,0}$ in $H$.  This representation $\rho$ is permissible in the sense of Definition 2.1.1 \cite{ANewYu}.
Under the assumptions on $p$ in this paper, ``permissible'' may be defined by:
 \bigskip
\begin{definition}\label{permissibledef} A representation of $H_x$ is {\bf permissible} if it is an irreducible  representation $(\rho ,V_\rho)$ of $H_x$ such that:
\begin{itemize}
\item[(1)] $\rho$ induces an irreducible (and hence supercuspidal) representation of $H$,
\item[(2)] the restriction of $\rho$ to $H_{x,0+}$ is a multiple of some character $\phi$ of $H_{x,0+}$, 
\item[(3)] $\phi$ is trivial on $H_{{\rm der},x,0+}$.
\end{itemize}
\end{definition}

The fact that this is equivalent to Definition 2.1.1 \cite{ANewYu} follows from  \cite[Lemma 3.2.1]{ANewYu} and \cite[Lemma 8.1]{MR1824988}.  Here, $H_{{\rm der},x,0+}$ denotes the pro-unipotent radical of the maximal parahoric subgroup $H_{{\rm der},x,0}$ of $H_{{\rm der}} = \bH_{\rm der}(F)$.

In \cite[\S2.5]{ANewYu}, we define a certain subgroup $K$ that is entirely analogous to the inducing subgroup (also denoted $K$) in Yu's construction.  In Theorem 2.8.1 and \S3.11, we construct from $\rho$ an irreducible representation $\kappa$ of $K$.  Then the regular supercuspidal representation associated to $\mu$ is the induced representation $\pi (\mu) = {\rm ind}_K^G(\kappa)$, where we use smooth induction with compact supports.

\subsection{The compatible families $\varepsilon^-$, $\varepsilon^+$, and $\varepsilon$}\label{sec:epsi}

In this section, we define and discuss three compatible families (denoted $\varepsilon^-$, $\varepsilon^+$, and $\varepsilon$) of characters  associated to a given tame, elliptic, regular pair $(\bT,\mu)$.

The family $\varepsilon^- = (\varepsilon^-_\theta)_{\theta \in \zeta}$ is a family of quadratic characters $$\varepsilon^-_\theta :T^\theta \to \{ \pm 1\}$$ associated to a $T$-orbit $\zeta$ of involutions $\theta$ of $G$ such that $\theta (\bT) = \bT$.  In the construction of \cite{ANewTasho}, we have a group $\mathsf{H}_x^\circ$ that is defined over  $\mathfrak{f}$ and is such that
$$\mathsf{H}_x^\circ (\mathfrak{F})= \bH (F^{\rm un})_{x,0:0+}.$$
Since $T$ normalizes $\bH (F^{\rm un})_{x,0}$ and $\bH (F^{\rm un})_{x,0+}$, there is a conjugation action of $T$ on $\mathsf{H}_x^\circ$, and an associated adjoint action $T$ on the Lie algebra 
$\mathfrak{h}_x$ of $\mathsf{H}_x^\circ (\mathfrak{f})$.  
There are similar actions on $\theta$-fixed points.
We define  
$$\varepsilon_\theta^-(t) = \det ({\rm Ad}(t)\, |\, \mathfrak{h}_x^\theta),$$ for $t\in T^\theta$.

The character $\varepsilon_\theta^-$ restricts and  factors to a character of $\mathsf{T}^\theta (\mathfrak{f}) = T_{0:0+}^\theta$.  This character is discussed in \cite[\S2]{MR1106911} and \cite[\S4.3]{ANewHM} and we will use some of its known properties.

The family $\varepsilon^+ = (\varepsilon^+_\theta)_{\theta \in \vartheta}$ is a family of quadratic characters $$\varepsilon^+_\theta :H_x^\theta \to \{ \pm 1\}$$ associated to an $H_x$-orbit $\vartheta$ of involutions  $\theta$  of $G$ such that $\theta (H_x) = H_x$.  The character $\varepsilon^+_\theta$ is the same as the character denoted by $\eta'_\theta$ in \cite[\S5.6]{MR2431732} and the character denoted by $\varepsilon_{{}_{H_x,\theta}}$ in \cite[\S2]{ANewHM}.

Finally, given $\zeta$ as above, we define the family $\varepsilon = (\varepsilon_\theta)_{\theta \in \zeta}$ to be the family of quadratic characters $$\varepsilon_\theta :T^\theta \to \{ \pm 1\}$$ 
given by $$\varepsilon_\theta (t) = \varepsilon^-_\theta (t) \varepsilon^+_\theta (t) .$$

The next result is proven in a special case in \cite[Lemma 7.10]{MR3027804}.  The same proof  carries over in general.  (See also Lemma 3.9 \cite{CZdist}.  The statement of the latter result is more restrictive, but the proof is essentially the same.)

\begin{lemma}\label{epconn} $\varepsilon^+_\theta | \left( (\mathsf{H}_x^{\circ \theta})^\circ (\mathfrak{f})\right) =1$. 
\end{lemma}

\begin{proof}
The essence of the proof lies in the following elementary observation.  Fix an index $j\in \{ 0,\dots , d\}$ and a positive number $s$.    Let $\mathsf{H}_j$ be the identity component of the $\theta$-fixed points of the $\mathfrak{f}$-group $\mathsf{G}_j^\circ = \mathsf{G}_j^\circ (\mathfrak{F}) = \bG^j(F^{\rm un})_{x,0:0+}$.   The group $\mathsf{H}_j$ has an adjoint action on the $\mathfrak{f}$-variety
$$\mathsf{V}_{j,s} = \mathsf{V}_{j,s}(\mathfrak{F}) = {\rm Lie}(\bG^j(F^{\rm un})^\theta)_{x,s:s+}$$ and we
consider the algebraic character
$$\chi_{j,s} (h) = \det\nolimits_{\mathfrak{F}} \left( \Ad (h)\ |\ \mathsf{V}_{j,s}  \right)$$ of $\mathsf{H}_j$.
We observe that $\mathsf{H}_j$ has a decomposition $\mathsf{H}_j = \mathsf{Z}_j\mathsf{H}_{j,{\rm der}}$ as a product of its center and its derived group, and, moreover, the restriction of $\chi_{j,s}$ to each of the factors $\mathsf{Z}_j$ and $\mathsf{H}_{j,{\rm der}}$ is obviously trivial.  
Consequently, $\chi_{j,s}$ is trivial.

Our claim now follows from the fact that $$\varepsilon_\theta^+ (h) = \prod_{i=1}^{d-1} \left(   \frac{\chi_{i+1, s_i}(h)}{\chi_{i, s_i}(h)}\right)^{(q-1)/2},$$ for $h\in  (\mathsf{H}_x^{\circ \theta})^\circ (\mathfrak{f})$.  Here, $q$ is the order of $\mathfrak{f}$.
(The latter expression for $\varepsilon_\theta^+$ is discussed in more detail in  the proofs of \cite[Lemma 7.10]{MR3027804} and Lemma 3.9 \cite{CZdist}.)\end{proof}

\subsection{Calculating $\left\langle \xi ,\rho_T^\mu\right\rangle^{\varepsilon^+}_{TH_{x,0}}$}

We need to compute 
$$\left\langle \xi ,\rho_T^\mu\right\rangle^{\varepsilon^+}_{TH_{x,0}} = \dim {\rm Hom}_{(TH_{x,0})^\theta}(\rho_T^\mu , \varepsilon^+_\theta),$$
where $\theta$ is an involution of $G$ in the $(TH_{x,0})$-orbit $\xi$.

(Implicit in the notation $\left\langle \xi ,\rho_T^\mu\right\rangle^{\varepsilon^+}_{TH_{x,0}}$  is that we are viewing $\varepsilon^+$ as a compatible family on $TH_{x,0}$, by restricting from $H_x$ to $TH_{x,0}$.)

\subsubsection{Some simple general facts}\label{sec:simpgen}

Let $\tau$ be an irreducible, smooth, complex representation of a compact-mod-center totally disconnected group $\mathscr{K}$.  Let $\mathscr{K}$ act on the set of its order two automorphisms by $$k\cdot \theta = {\rm Int}(k)\circ\theta\circ {\rm Int}(k)^{-1},$$ where ${\rm Int}(k)(h) = khk^{-1}$.

Let $\xi$ be a $\mathscr{K}$-orbit.  Let $\chi = (\chi_\theta)_{\theta\in \xi}$ be a family of characters, such that $\chi_\theta$ is a character of $\mathscr{K}^\theta$ and $\chi_{k\cdot \theta}(klk^{-1}) =\chi_\theta (l)$ for all $k\in \mathscr{K}$ and $l\in \mathscr{K}^\theta$.

We now list some useful facts that are both standard and elementary.

\begin{fact}\label{firstfact} $\mathscr{K}^{k\cdot \theta} = k\mathscr{K}^\theta k^{-1}$, for all order two automorphisms $\theta$ of $\mathscr{K}$ and all $k\in \mathscr{K}$.\end{fact}

In some cases of interest, we will have a group $\mathscr{K}$ that is a closed subgroup of a larger group $\mathscr{G}$ and we will have an order two automorphism $\theta$ of $\mathscr{G}$ that does not necessarily stabilize $\mathscr{K}$.  In such cases, we define $\mathscr{K}^\theta = \mathscr{K}\cap \mathscr{G}^\theta$, and we caution that
 when $g\in \mathscr{G}$ then we may have
 $$\mathscr{K}^{g\cdot \theta}  =  g\left((g^{-1}\mathscr{K}g)^\theta \right) g^{-1} \ne g\mathscr{K}^\theta  g^{-1}.$$

\medskip
\begin{fact}\label{secondfact} For each $k\in \mathscr{K}$ and $\theta\in \xi$, the map  that sends $\Lambda\in {\rm Hom}_{\mathscr{K}^\theta}(\tau,\chi_\theta)$ to the linear form ${}^k\Lambda$ defined by ${}^k\Lambda(v) = \Lambda(\tau(k)^{-1}v)$ is an isomorphism
${\rm Hom}_{\mathscr{K}^\theta}(\tau,\chi_\theta)\cong {\rm Hom}_{\mathscr{K}^{k\cdot \theta}}(\tau,\chi_{k\cdot\theta})$.\end{fact}

\begin{fact}\label{isotypy}
Suppose $\mathscr{S}$ is a subgroup of $\mathscr{K}^\theta$ such that $\tau |\mathscr{S} = \chi_{\mathscr{S}}\cdot {\rm Id}$ for some character $\chi_{\mathscr{S}}$ of $\mathscr{S}$.  Suppose also that ${\rm Hom}_{\mathscr{K}^\theta}(\tau,\chi_\theta)$ is nonzero.  Then $\chi_{\mathscr{S}} = \chi_\theta |\mathscr{S}$.
\end{fact}

\begin{fact}\label{lastfact}
$$\dim {\rm Hom}_{\mathscr{K}^\theta}(\tau,\chi_\theta) = \oint_{\mathscr{K}^\theta/ (Z(\mathscr{K})\cap \mathscr{K}^\theta)}{\rm tr}(\tau (k))\ \chi(k)^{-1}\ dk,$$ where $Z(\mathscr{K})$ denotes the center of $\mathscr{K}$ and $\oint$ signifies integration with respect to the Haar measure that gives $\mathscr{K}^\theta/ (Z(\mathscr{K})\cap \mathscr{K}^\theta)$ measure one.
\end{fact}

\subsubsection{Applications of the general facts}

We have already used Fact \ref{firstfact} extensively, for example, in \S\ref{sec:indweights}.  Our main purpose in mentioning it in \S\ref{sec:simpgen} is to caution the reader that the fact does not apply directly to subgroups that are not stabilized by a given involution.

We apply the remaining facts when $\left\langle \xi ,\rho_T^\mu\right\rangle^{\varepsilon^+}_{TH_{x,0}} $ is nonzero, $\theta\in \xi$, and
\begin{itemize}
\item $\mathscr{K} = TH_{x,0}$,
\item $\mathscr{S} = (ZH_{x,0+})^\theta$, where $Z =\bZ (F)$ and $\bZ$ is the center of $\bG$,
\item $\chi_{\mathscr{S}} = \mu^\sharp |(ZH_{x,0+})^\theta$,
\item $\tau = \rho_T^\mu$,
\item $\chi_\theta =\varepsilon_\theta^+$.
\end{itemize}
According to Lemma 2.9 and Proposition 2.12 in \cite{MR2431732}, we have 
$(ZH_{x,0+})^\theta = Z^\theta H_{x,0+}^\theta$.

Fact \ref{secondfact} implies that the definition
$$\left\langle \xi ,\rho_T^\mu\right\rangle^{\varepsilon^+}_{TH_{x,0}} = \dim {\rm Hom}_{(TH_{x,0})^\theta}(\rho_T^\mu , \varepsilon^+_\theta),$$
for $\theta\in\xi$, does not depend on the choice of $\theta$ in $\xi$.

Fact \ref{isotypy} is used in the proof of:

\begin{lemma}\label{boringlemma}
The restriction of
$\rho_T^\mu$ to $Z  H_{x,0+}$  is a multiple of the character $\mu^\sharp |(ZH_{x,0+})$.
If $\left\langle \xi ,\rho_T^\mu\right\rangle^{\varepsilon^+}_{TH_{x,0}} $ is nonzero and $\theta\in \xi$ then the restriction of $\mu^\sharp$ to $(ZH_{x,0+})^\theta = Z^\theta H_{x,0+}^\theta$  is trivial.
\end{lemma}

\begin{proof}
We first show $$\rho_T^\mu |(Z  H_{x,0+}) = \mu^\sharp |(ZH_{x,0+}) \cdot {\rm Id}.$$ 
If $z\in Z$, then, in the notation of \cite[\S2.6]{ANewTasho}, we have $\rho_T^\mu (z) = \mu (z)\ {\rm Int}(z) = \mu(z){\rm Id} = \mu^\sharp (z){\rm Id}$.  On the other hand, if $h\in H_{x,0+}$ and $(\mu_-,\mu_+)$ is a weak factorization of $\mu$ then
$\rho_T^\mu (h) = \mu_+(h)\rho_T^{\mu_-}(h) = \mu_+(h){\rm Id} = \mu^\sharp (h){\rm Id}$.

Now, using the fact that $\varepsilon_\theta^+ |  (ZH_{x,0+})^\theta$ is trivial,  
we can apply Fact \ref{isotypy} to deduce  
 $\mu^\sharp |(ZH_{x,0+})^\theta =1$.
\end{proof}

Fact \ref{lastfact} and the triviality of $\varepsilon^+_\theta | (ZH_{x,0+})^\theta$ imply
\begin{eqnarray*}
\left\langle \xi ,\rho_T^\mu\right\rangle^{\varepsilon^+}_{TH_{x,0}} 
&=& \dim {\rm Hom}_{(TH_{x,0})^\theta}(\rho_T^\mu , \varepsilon^+_\theta)\\
&=& \frac{1}{|\mathscr{K}^\theta  |} \sum_{k\in \mathscr{K}^\theta}
{\rm tr}(\rho_T^\mu(k))\ \varepsilon^+_\theta (k).\end{eqnarray*}

\subsubsection{Jordan decompositions}\label{sec:Jor}

Once again, we fix $\theta\in \xi$ such that $\left\langle \xi ,\rho_T^\mu\right\rangle^{\varepsilon^+}_{TH_{x,0}} \ne 0$ and we let $$\mathscr{K}^\theta = (TH_{x,0})^\theta/(ZH_{x,0+})^\theta .$$
As we have just observed, we have
$$
\left\langle \xi ,\rho_T^\mu\right\rangle^{\varepsilon^+}_{TH_{x,0}} 
= \frac{1}{|\mathscr{K}^\theta  |} \sum_{k\in \mathscr{K}^\theta}
{\rm tr}(\rho_T^\mu(k))\ \varepsilon^+_\theta (k).$$

The elements of $\mathscr{K}^\theta$ have unique ``Jordan decompositions''  in a natural sense.  This can be explained from various perspectives, but perhaps the simplest is as follows.

Suppose $k\in \mathscr{K}^\theta$.  Then the group $\langle k\rangle$ generated by $k$ is a finite abelian group.  If $\langle k\rangle_p$ is the (unique) Sylow $p$-subgroup of $\langle k\rangle$ and $\langle k\rangle_{p'}$ is the product of the other Sylow  subgroups, then we have a direct product decomposition
$$\langle k\rangle = \langle k\rangle_p\times \langle k\rangle_{p'}.$$
 Our desired Jordan decomposition on $\mathscr{K}^\theta$ is $k = k_sk_u = k_uk_s$, where $k_u$ is the $p$-component of $k$ and  $k_s$ is the $p'$-component.  (To understand the connection between this decomposition and other notions of Jordan decompositions, we refer to \cite{MR2408311}.)

Let
\begin{eqnarray*}
\mathscr{K}^\theta_s&=&\{ k\in \mathscr{K}^\theta\ :\ k= k_s\},\\
\mathscr{K}^\theta_u&=&\{ k\in \mathscr{K}^\theta\ :\ k= k_u\}.
\end{eqnarray*}
Then $$\sum_{k\in \mathscr{K}^\theta}
{\rm tr}(\rho_T^\mu(k))\ \varepsilon^+_\theta (k)
=
\sum_{\substack{s\in \mathscr{K}^\theta_s \\   u\in \mathscr{K}^\theta_u \\   su=us}}
{\rm tr}(\rho_T^\mu(su))\ \varepsilon^+_\theta (s).
$$
Note that since $\varepsilon_\theta^+$ has exponent two, it must be trivial on all elements of $\mathscr{K}_u^\theta$ since they have odd order.

We therefore have the formula
$$
\left\langle \xi ,\rho_T^\mu\right\rangle^{\varepsilon^+}_{TH_{x,0}} 
= \frac{1}{|\mathscr{K}^\theta  |} \sum_{\substack{s\in \mathscr{K}^\theta_s \\   u\in \mathscr{K}^\theta_u \\   su=us}}
{\rm tr}(\rho_T^\mu(su))\ \varepsilon^+_\theta (s).$$

\subsubsection{The character of $\rho_T^\mu$}

Proposition 3.2.1 \cite{ANewTasho} establishes a character formula
$${\rm trace} (\rho_T^\mu(su)) = \frac{(-1)^{\ell(w)}}{|\mathsf{M}_s(\mathfrak{f})|}\sum_{h\in \mathsf{H}_x^\circ (\mathfrak{f})}{}^h\dot\mu(s)\ \hat\mu (u)\ Q_{h\mathsf{T}h^{-1}}^{\mathsf{M}_s}(u)$$
for the character of $\rho_T^\mu$ at an element $k\in TH_{x,0}$ whose topological Jordan decomposition, in the sense of Definition 3.1.1 \cite{ANewTasho}, is $k=su$.

The various notations involved in this formula involve some subtleties that are discussed fully in \S3.2 \cite{ANewTasho}.  We now give a brief description of the objects involved in the formula.  At the outset, we should stress that for $k\in H_{x,0}$ our character formula reduces to Deligne-Lusztig's character formula (\cite[Theorem 4.2]{MR0393266}) for the representation of $\mathsf{H}_x^\circ (\mathfrak{f}) = H_{x,0:0+}$ associated to $\rho_T^\mu |H_{x,0}$.

For general $k\in TH_{x,0}$, we can consider the image $\bar k$ of $k$ in $H/A_H$, where $A_H$ is the split component of the center of $H$.  This element $\bar k$ has a topological Jordan decomposition $\bar k = \bar s \bar u$.  We can lift $\bar s$ and $\bar u$ to elements $s$ and $u$ in $TH_{x,0}$ such that $u$ is topologically unipotent and $su=us$.
This decomposition $k = su=us$ is described in Proposition 3.1.2 \cite{ANewTasho}, including the extent to which the decomposition is unique.

We are interested in averaging ${\rm trace} (\rho_T^\mu(k)) $ over the elements $k$ that lie in $(TH_{x,0})^\theta$, where $\theta\in \xi$ and $\left\langle \xi ,\rho_T^\mu\right\rangle^{\varepsilon^+}_{TH_{x,0}} \ne 0$.  But this trace only depends on the image of $k$ in $\mathscr{K}^\theta = (TH_{x,0})^\theta/(ZH_{x,0+})^\theta$.  It is elementary to see that the image of the topological Jordan decomposition of $k$ in $\mathscr{K}^\theta$ coincides with the Jordan decomposition discussed above in \S\ref{sec:Jor}.

The function $\dot\mu :TH_{x,0}\to \C$ is a canonical extension of $\mu$ that vanishes outside of $TH_{x,0+}$.   If $\mu$ has depth zero then $\dot\mu |TH_{x,0+}$ is the inflation of $\mu$ over $T$.  Otherwise, $\dot\mu |TH_{x,0+}$ is defined using a weak factorization of $\mu$.
For $h\in H_{x,0}$, we take ${}^h \dot\mu (s) = \dot\mu (h^{-1}sh)$.  Since $h\mapsto {}^h \dot\mu (s)$ factors through $\mathsf{H}^\circ_x(\mathfrak{f})$, we may view ${}^h \dot\mu (s)$ as being defined for $h$ in $\mathsf{H}^\circ_x(\mathfrak{f})$.

The function $\hat\mu$ is an extension of $\mu$ to the set $T\mathscr{U}$, where $\mathscr{U}$ is the set of topologically unipotent elements in $H_{x,0}$.  If $\mu$ has depth zero then $\hat\mu$ is the inflation of $\mu$ over $\mathscr{U}$.  Otherwise, it is defined using a weak factorization.

The $\mathfrak{f}$-group $\mathsf{M}_s$ is the connected centralizer of $s$ in $\mathsf{H}_x^\circ$, in the sense of \cite{ANewYu}.  Let $\mathsf{T}$ be the maximal torus in $\mathsf{H}_x^\circ$ determined by $\bT$.  When ${}^h\dot\mu (s)$ is nonzero, it must be the case that $h\mathsf{T}h^{-1}$ is a maximal $\mathfrak{f}$-torus in $\mathsf{M}_s$.  Therefore, there is an associated Green function $Q_{h\mathsf{T}h^{-1}}^{\mathsf{M}_s}$ defined on the unipotent set in $\mathsf{M}_s$.  In our character formula, $Q_{h\mathsf{T}h^{-1}}^{\mathsf{M}_s} (u)$ denotes $Q_{h\mathsf{T}h^{-1}}^{\mathsf{M}_s} (\bar u)$.

As discussed in \S2.6 \cite{ANewTasho}, the sign $(-1)^{\ell (w)}$ is the standard sign factor that occurs in the Deligne-Lusztig theory.

Further details regarding the material in this section are in \cite{ANewTasho}.

\subsubsection{$p$-unipotent elements}

\begin{lemma}
Every element of $\mathscr{K}_u^\theta$, viewed as a coset in $$(TH_{x,0})^\theta/(ZH_{x,0+})^\theta,$$ has a topologically unipotent representative in $H_{x,0}^\theta$.  
There is a well-defined bijection  $$\mathscr{K}_u^\theta\leftrightarrow \{ \text{unipotent elements in }\mathsf{H}_x^\circ (\mathfrak{f})^\theta\}$$ that is given as follows:  the image of $k\in \mathscr{K}_u^\theta$ in $\mathsf{H}_x^\circ (\mathfrak{f})^\theta$  is just the reduction modulo $H_{x,0+}$  of a coset representative of $k$ in $H_{x,0}$.
\end{lemma}

\begin{proof}
Consider the  group $\mathscr{G}= \mathscr{K}^\theta$ and its subgroups $\mathscr{A} = T^\theta/(ZT_{0+})^\theta$ and $\mathscr{B} = H_{x,0}^\theta/ (ZH_{x,0+})^\theta$.   

What is most relevant for our argument now is that $\mathscr{G}$ is finite group 
 with subgroups $\mathscr{A}$ and $\mathscr{B}$ such that $\mathscr{A}$ has prime-to-$p$ order and it normalizes $\mathscr{B}$.  
Let $S$ be the set of Sylow $p$-subgroups of the (semidirect) product $\mathscr{A}\mathscr{B}$.  Let $\mathscr{C}$ be any Sylow $p$-subgroup of $\mathscr{B}$.  Then $\mathscr{C}$ must be an element of $S$.  But every  other element of $S$ must have the form $ab\mathscr{C}b^{-1}a^{-1}$ for some $a\in \mathscr{A}$ and $b\in \mathscr{B}$.  
It follows that all of the elements of $S$ are subgroups of $\mathscr{B}$, and hence every element of $\mathscr{K}^\theta_u$ has a representative in $H_{x,0}^\theta$.  (Below we show that we can choose a topologically unipotent representative.)

The image of $H_{x,0}^\theta$ in $\mathscr{K}^\theta$ is isomorphic to $H_{x,0}^\theta /(Z_0H_{x,0+})^\theta$.
We therefore have a bijection between $\mathscr{K}^\theta_u$ and the set of elements of $H_{x,0}^\theta /(Z_0H_{x,0+})^\theta$ whose order is a power of $p$.

Now consider the group $\mathscr{D} = H_{x,0}^\theta / H_{x,0+}^\theta$ and the subgroup $\mathscr{E}=(Z_0H_{x,0+})^\theta/H_{x,0+}^\theta$.
Then $\mathscr{D}$ is a finite group and $\mathscr{E}$ is a central subgroup of prime-to-$p$ order.  So the projection $\mathscr{D}\to \mathscr{D}/\mathscr{E}$ gives a bijection between the sets of elements in $\mathscr{D}$ and $\mathscr{D}/\mathscr{E}$ whose order is a power of $p$.  This yields a bijection between $\mathscr{K}^\theta_u$ and the set of elements of $H_{x,0}^\theta /H_{x,0+}^\theta$ whose order is a power of $p$.

Clearly, the  group $\mathscr{D}$ is just the image of $H_{x,0}^\theta$ in $\mathsf{H}_x^\circ (\mathfrak{f})^\theta$.  But we claim that, in fact, $\mathscr{D} = \mathsf{H}_x^\circ (\mathfrak{f})^\theta$ (or, more precisely, $\mathscr{D}$ surjects onto $\mathsf{H}_x^\circ (\mathfrak{f})^\theta$).  Indeed, the elements of $\mathsf{H}_x^\circ (\mathfrak{f})^\theta$ are just the cosets in $H_{x,0:0+}$ that have a representative $h\in H_{x,0}$ such that $h^{-1}\theta (h)\in H_{x,0+}$.  Given such an $h$, according to \cite[Proposition 2.12]{MR2431732}, we can choose $k\in H_{x,0+}$ such that $h^{-1}\theta (h) = k\theta(k)^{-1}$.  So $hH_{x,0+}$ contains an element of $H_{x,0}^\theta$, namely, $hk$.  Thus, $\mathscr{D} = \mathsf{H}_x^\circ (\mathfrak{f})^\theta$.

We now have a bijection between $\mathscr{K}^\theta_u$ and the set of elements of $\mathsf{H}_x^\circ (\mathfrak{f})^\theta$ whose order is a power of $p$.
But since the Sylow $p$-subgroups of $\mathsf{H}_x^\circ (\mathfrak{f})$ are just its maximal unipotent subgroups (see \cite[Proposition 3.19]{MR1118841}), the elements whose order is a power of $p$ are precisely  the unipotent elements.

It only remains to show 
that every element of $\mathscr{K}^\theta_u$ has a topologically unipotent lift in $H_{x,0}^\theta$.
Now suppose $h\in H_{x,0}^\theta$ has unipotent image $\mathsf{H}_x^\circ (\mathfrak{f})$ and let $h= h_sh_u$ be its (unique) topological Jordan decomposition.  (See \cite{MR2408311}.)  Applying $\theta$ to the decomposition and using uniqueness shows that the components $h_s$ and $h_u$ are necessarily $\theta$-fixed.  According to Remark 1.9 \cite{MR2408311}, the element $h_s$ has finite prime-to-$p$ order.  Since $h_s$ commutes with $h_u$, we see that the image of $h_s$ in $\mathsf{H}_x^\circ (\mathfrak{f})$ must be trivial.  So $h_u$ must be a topologically unipotent element in $H_{x,0}^\theta$ that has the same image as $h$.  It follows that every element of $\mathscr{K}^\theta_u$ has a topologically unipotent lift in $H_{x,0}^\theta$.
\end{proof}


\begin{lemma}
If $k\in \mathscr{K}_u^\theta$ and $h\in H_{x,0}^\theta$ is a topologically unipotent lift of $k$ then
$\hat\mu (h) =1$.
\end{lemma}

\begin{proof}
Suppose we are given $k$ and $h$.  Since the image of $h$ in $\mathsf{H}_x^\circ (\mathfrak{f})$ is a commutator, we may choose a commutator $j$ in $H_{x,0}^\theta$ and $\ell \in H^\theta_{x,0+}$ such that $h= j\ell$.

If $(\mu_-,\mu_+)$ is any weak factorization of $\mu$ then $\hat\mu (h) = \mu_+(h) = \mu_+(j)\mu_+(\ell) =1$.  Indeed, $\mu_+(j)=1$ since $j$ a commutator and $\mu_+$ is a character of $H$, while $\mu_+ (\ell )=1$ according to Lemma \ref{boringlemma}.

Suppose $u\in \mathscr{K}_u^\theta$.  Then $u$ may be viewed as a unipotent element in $\mathsf{H}_x^\circ (\mathfrak{f})$ and hence as a commutator in $\mathsf{H}_x^\circ (\mathfrak{f})$.  We can lift $u$ to a commutator $h\in H_{x,0}$.
The function $\hat\mu$ is defined on the set of topologically unipotent elements in $H_{x,0}$ 
Then $\hat\mu (h) = \mu_+(h)$, where $(\mu_-,\mu_+)$ is any weak factorization of $\mu$.  But since $h$ is a commutator and $\mu_+$ is a character of $H$, we have $\mu_+(h)=1$.
\end{proof}

\subsubsection{$p$-semisimple elements}

The character formula for the character of $\rho_T^\mu$, together with the theory in the previous section, yields:
$$\left\langle \xi ,\rho_T^\mu\right\rangle^{\varepsilon^+}_{TH_{x,0}} 
= \frac{(-1)^{\ell(w)}}{|\mathscr{K}^\theta |} 
\!\sum_{\substack{s\in \mathscr{K}'^\theta_s \\   h\in \mathsf{H}_x^\circ (\mathfrak{f}) }}\!\!\!
 \frac{{}^h\dot\mu(s)\varepsilon^+_\theta (s)}{|\mathsf{M}_s(\mathfrak{f})|} \!\!\!\sum_{\substack{u\in \mathsf{M}_s(\mathfrak{f})^\theta \\ u \text{ unipotent} }}\!\! Q_{h\mathsf{T}h^{-1}}^{\mathsf{M}_s}(u).$$
 where $$\mathscr{K}'^\theta_s = ({\rm Int}(H_{x,0})(TH_{x,0+}))^\theta/(ZH_{x,0+})^\theta.$$
 
 Implicit in this formula is the fact that ${}^h\dot\mu (s)$ is well-defined for $s\in \mathscr{K}'^\theta_s$.  The first issue is that ${}^h\dot\mu$ is defined on $TH_{x,0}$, whereas the elements $s$ are cosets of $(ZH_{x,0+})^\theta$.  But, according to Lemma \ref{boringlemma}, $\mu^\sharp |(ZH_{x,0+})^\theta =1$, which implies that ${}^h\dot\mu$ is constant on the relevant cosets.
 Clearly, conjugation by $h\in \mathsf{H}_x^\circ (\mathfrak{f})$ also preserves cosets of $(ZH_{x,0+})^\theta$.  So $h^{-1}sh$ and ${}^h\dot\mu (s) = \dot\mu (h^{-1}sh)$ have obvious meanings.
 
The term ${}^h\dot\mu (s)$ is nonzero precisely when $h^{-1}sh$ lies in 
$$\mathscr{T}^{h^{-1}\cdot \theta } = (TH_{x,0+})^{h^{-1}\cdot \theta } /(ZH_{x,0+})^{h^{-1}\cdot \theta }
 \cong T^{h^{-1}\cdot \theta }/(ZT_{0+})^{h^{-1}\cdot \theta } .$$

\subsubsection{Generalizing a computation of Lusztig}

\begin{proposition}\label{Luszt}
$$\left\langle \xi ,\rho_T^\mu\right\rangle^{\varepsilon^+}_{TH_{x,0}} 
=  \#  \left( (TH_{x,0+})\bs \Xi_{\mathsf{T},\mu}/  (TH_{x,0})^\theta \right),$$
where
$$\Xi_{\mathsf{T},\mu} = \left\{ h\in TH_{x,0} \ :\  (h\cdot\theta ) (\mathsf{T}) = \mathsf{T},\ \mu\varepsilon_{h\cdot \theta}|T^{h\cdot\theta} =1\right\}.$$
\end{proposition}

\noindent
Proposition \ref{Luszt} will be proven at the end of this section.  The core of the proof is contained in:

\begin{lemma}\label{almostprop}
$$\left\langle \xi ,\rho_T^\mu\right\rangle^{\varepsilon^+}_{TH_{x,0}} 
=  \sigma (\mathsf{H}_x^\circ)
\sum_{h\in (TH_{x,0+})\bs \Xi_{\mathsf{T},\mu}/  (TH_{x,0})^\theta } 
\sigma (Z_{\mathsf{H}^\circ_x}((\mathsf{T}^{h\cdot\theta})^\circ)).$$
\end{lemma}

\begin{proof}
Applying the discussion in the previous section, we have
$$\left\langle \xi ,\rho_T^\mu\right\rangle^{\varepsilon^+}_{TH_{x,0}} 
= \frac{(-1)^{\ell(w)}}{|\mathscr{K}^\theta  |} 
\!\sum_{\substack{h\in \mathsf{H}_x^\circ (\mathfrak{f}) \\s\in h\mathscr{T}^{h^{-1}\cdot \theta }  h^{-1}}}\!\!\!
 \frac{{}^h\dot\mu(s)\varepsilon^+_\theta (s)}{|\mathsf{M}_s(\mathfrak{f})|} \!\!\!\sum_{\substack{u\in \mathsf{M}_s(\mathfrak{f})^\theta \\ u \text{ unipotent} }}\!\! Q_{h\mathsf{T}h^{-1}}^{\mathsf{M}_s}(u).$$

The unipotent sum may be expressed using Lusztig's Theorem 3.4 \cite{MR1106911} as $$\sum_{\substack{u\in \mathsf{M}_s(\mathfrak{f})^\theta \\ u \text{ unipotent} }}\!\! Q_{h\mathsf{T}h^{-1}}^{\mathsf{M}_s}(u)= \frac{\sigma (\mathsf{T})}{|\mathsf{T}(\mathfrak{f})|}
\!\!\!\!\sum_{\substack{g\in \mathsf{M}_s(\mathfrak{f})\\ (g\cdot \theta )(h\mathsf{T}h^{-1}) =h\mathsf{T}h^{-1}
}} \sigma (Z_{\mathsf{M}_s}(((h\mathsf{T}h^{-1})^{g\cdot \theta})^\circ)).$$

Combining this with the fact that  $(-1)^{\ell(w)}\sigma (\mathsf{T})= \sigma (\mathsf{H}_x^\circ)$, we may express $\left\langle \xi ,\rho_T^\mu\right\rangle^{\varepsilon^+}_{TH_{x,0}} $ as
$$\frac{\sigma (\mathsf{H}_x^\circ)}{|\mathscr{K}^\theta  | |\mathsf{T}(\mathfrak{f})|} 
\!\!\!\!\sum_{\substack{h\in \mathsf{H}_x^\circ (\mathfrak{f})\\ 
g\in \mathsf{M}_s(\mathfrak{f})\\ (g\cdot \theta )(h\mathsf{T}h^{-1}) =h\mathsf{T}h^{-1}\\ s\in h\mathscr{T}^{h^{-1}\cdot \theta } h^{-1} }}\!\!\!\!\!\!\!
 \frac{{}^h\dot\mu(s)\varepsilon^+_\theta (s)}{|\mathsf{M}_s(\mathfrak{f})|} 
\sigma (Z_{\mathsf{M}_s}(((h\mathsf{T}h^{-1})^{g\cdot \theta})^\circ)) .$$

Applying the change of variables $h\mapsto gh$, this becomes
$$\frac{\sigma (\mathsf{H}_x^\circ)}{|\mathscr{K}^\theta  | |\mathsf{T}(\mathfrak{f})|} 
\!\!\!\!\!\!\sum_{\substack{h\in \mathsf{H}_x^\circ (\mathfrak{f})\\
g\in \mathsf{M}_s(\mathfrak{f})\\ (h^{-1}\cdot \theta )(\mathsf{T}) =\mathsf{T}\\  s\in h\mathscr{T}^{h^{-1}\cdot \theta } h^{-1} }}\!\!\!\!\!\!\!\!\!
 \frac{{}^h\dot\mu(s)\varepsilon^+_\theta (s)}{|\mathsf{M}_s(\mathfrak{f})|} 
\sigma (Z_{\mathsf{M}_s}(((h\mathsf{T}h^{-1})^{\theta})^\circ)) .$$

Now we notice that $g$ only appears as an index in the sum, and hence
$$\left\langle \xi ,\rho_T^\mu\right\rangle^{\varepsilon^+}_{TH_{x,0}} 
= \frac{\sigma (\mathsf{H}_x^\circ)}{|\mathscr{K}^\theta  | |\mathsf{T}(\mathfrak{f})|} 
\!\!\!\!\sum_{\substack{ h\in \mathsf{H}_x^\circ (\mathfrak{f})\\ (h^{-1}\cdot \theta )(\mathsf{T}) =\mathsf{T}\\ s\in h\mathscr{T}^{h^{-1}\cdot \theta } h^{-1} }}\!\!\!\!\!\!\!
{}^h\dot\mu(s)\varepsilon^+_\theta (s)
\sigma (Z_{\mathsf{M}_s}(((h\mathsf{T}h^{-1})^{\theta})^\circ)) .$$

The restriction of $\varepsilon_\theta^-$ to $T_0^\theta$ factors 
 Lusztig's function on $\mathsf{T}(\mathfrak{f})^\theta$:
$$\varepsilon_{\mathsf{T},\theta} (t) = \sigma( Z_{\mathsf{H}_x^\circ} ((\mathsf{T}^\theta)^\circ))\ \sigma (Z_{\mathsf{M}_t}((\mathsf{T}^\theta)^\circ)).$$
We note that $$\varepsilon_{h\mathsf{T}h^{-1},\theta} (s) = \varepsilon_{\mathsf{T},h^{-1}\cdot\theta} (h^{-1}sh)$$
and, similarly, 
$$\varepsilon^+_{\theta} (s) = \varepsilon^+_{h^{-1}\cdot\theta} (h^{-1}sh).$$
So
$$\left\langle \xi ,\rho_T^\mu\right\rangle^{\varepsilon^+}_{TH_{x,0}} 
= \frac{\sigma (\mathsf{H}_x^\circ)}{|\mathscr{K}^\theta  | |\mathsf{T}(\mathfrak{f})|} 
\!\!\!\!\sum_{\substack{h\in \mathsf{H}_x^\circ (\mathfrak{f})\\ (h^{-1}\cdot \theta )(\mathsf{T}) =\mathsf{T} \\  s\in h\mathscr{T}^{h^{-1}\cdot \theta } h^{-1}}}\!\!\!\!\!\!\!
{}^h\dot\mu(s){}^h\varepsilon_{h^{-1}\cdot\theta} (s)
\sigma (Z_{\mathsf{H}^\circ_x}((\mathsf{T}^{h^{-1}\cdot\theta})^\circ)) .$$

The next step is to replace $h$ with $h^{-1}$ and to let $t= shs^{-1}$:
$$\left\langle \xi ,\rho_T^\mu\right\rangle^{\varepsilon^+}_{TH_{x,0}} 
= \frac{\sigma (\mathsf{H}_x^\circ)}{|\mathscr{K}^\theta  | |\mathsf{T}(\mathfrak{f})|} 
\!\!\!\!\sum_{\substack{h\in \mathsf{H}_x^\circ (\mathfrak{f})\\ (h \cdot \theta )(\mathsf{T}) =\mathsf{T} \\  t\in \mathscr{T}^{h\cdot \theta} }}\!\!\!\!\!\!\!
\dot\mu(t)\varepsilon_{h\cdot\theta} (t)
\sigma (Z_{\mathsf{H}^\circ_x}((\mathsf{T}^{h\cdot\theta})^\circ)) .$$

We can now evaluate the sum over $t$, which is just the average of a character of $\mathscr{T}^{h\cdot\theta}$.
The value of the sum is $|\mathscr{T}^{h\cdot\theta} |$ when $\mu\varepsilon_{h\cdot \theta}| T^{h\cdot\theta} =1$ and it vanishes otherwise.

Let
$$\overline{\Xi}_{\mathsf{T},\mu} = \left\{ h\in \mathsf{H}_x^\circ (\mathfrak{f})\ :\  (h\cdot\theta ) (\mathsf{T}) = \mathsf{T},\ \mu\varepsilon_{h\cdot \theta}|T^{h\cdot\theta} =1\right\}.$$
Then
$$\left\langle \xi ,\rho_T^\mu\right\rangle^{\varepsilon^+}_{TH_{x,0}} 
= \frac{\sigma (\mathsf{H}_x^\circ) }{|\mathscr{K}^\theta  | |\mathsf{T}(\mathfrak{f})|} 
\sum_{h\in \overline{\Xi}_{\mathsf{T},\mu}}  |\mathscr{T}^{h\cdot \theta} |\ 
\sigma (Z_{\mathsf{H}^\circ_x}((\mathsf{T}^{h\cdot\theta})^\circ)) .$$

The group $\mathsf{T}(\mathfrak{f})$ acts by left translations on $\Xi_{\mathsf{T},\mu}$ and we have
$$\left\langle \xi ,\rho_T^\mu\right\rangle^{\varepsilon^+}_{TH_{x,0}} 
= \frac{\sigma (\mathsf{H}_x^\circ) }{|\mathscr{K}^\theta  |} 
\sum_{h\in \mathsf{T}(\mathfrak{f})\bs \overline{\Xi}_{\mathsf{T},\mu}}  |\mathscr{T}^{h\cdot \theta} |\ 
\sigma (Z_{\mathsf{H}^\circ_x}((\mathsf{T}^{h\cdot\theta})^\circ)) .$$
We can re-express this using the bijection
$$\mathsf{T}(\mathfrak{f}) \bs \mathsf{H}_x^\circ (\mathfrak{f}) \cong
(TH_{x,0+})\bs (TH_{x,0}).$$
Taking
$$\Xi_{\mathsf{T},\mu} = \left\{ h\in TH_{x,0} \ :\  (h\cdot\theta ) (\mathsf{T}) = \mathsf{T},\ \mu\varepsilon_{h\cdot \theta}|T^{h\cdot\theta} =1\right\},$$
we have
$$\left\langle \xi ,\rho_T^\mu\right\rangle^{\varepsilon^+}_{TH_{x,0}} 
= \frac{\sigma (\mathsf{H}_x^\circ) }{|\mathscr{K}^\theta  |} 
\sum_{h\in (TH_{x,0+})\bs  \Xi_{\mathsf{T},\mu}}  |\mathscr{T}^{h\cdot \theta} |\ 
\sigma (Z_{\mathsf{H}^\circ_x}((\mathsf{T}^{h\cdot\theta})^\circ)) .$$

The group $\mathscr{K}^\theta$ acts  on $(TH_{x,0+})\bs \Xi_{\mathsf{T},\mu}$ by right translations.  The stabilizer of a given coset $(TH_{x,0+}) h$ is $h^{-1}\mathscr{T}^{h\cdot \theta } h$.
Therefore
$$\left\langle \xi ,\rho_T^\mu\right\rangle^{\varepsilon^+}_{TH_{x,0}} 
=  \sigma (\mathsf{H}_x^\circ)
\sum_{h\in (TH_{x,0+})\bs \Xi_{\mathsf{T},\mu}/  (TH_{x,0})^\theta } 
\sigma (Z_{\mathsf{H}^\circ_x}((\mathsf{T}^{h\cdot\theta})^\circ)) .$$
\end{proof}

\begin{proof}[Proof of Proposition \ref{Luszt}]  
In light of Lemma \ref{almostprop}, it suffices to show that if $h\in  \Xi_{\mathsf{T},\mu}$ then 
$$\sigma (Z_{\mathsf{H}^\circ_x}((\mathsf{T}^{h\cdot\theta})^\circ)) = \sigma (\mathsf{H}_x^\circ).$$
We may as well assume $h=1$.
So we are assuming 
$\theta (\mathsf{T}) = \mathsf{T}$ and $\mu\varepsilon_{\theta}|T^{\theta} =1$ and we must show
$$\sigma (Z_{\mathsf{H}^\circ_x}((\mathsf{T}^{\theta})^\circ)) = \sigma (\mathsf{H}_x^\circ).$$

According to Lemma 10.5 \cite{MR1106911}, it suffices to show that
$\mathsf{T}$ lies in the set $\mathscr{J}_\theta$ of maximal tori of the form $\mathsf{B}\cap \theta (\mathsf{B})$ for some Borel subgroup $\mathsf{B}$.  This is equivalent to showing that $\theta$ does not fix any roots (or, equivalently, any coroots) of $(\mathsf{H}_x^\circ,\mathsf{T})$, according to \cite[\S10.1(a)]{MR1106911}.

So, in order to obtain a contradiction, we assume that $a$ is $\theta$-fixed root.  We observe now that since $\varepsilon_\theta$ and $\mu\varepsilon_{\theta}$ have trivial restrictions to $T^\theta_{0+}$, it must be the case that $\mu |T^\theta_0$ factors to a character $\bar\mu_\theta$ of $\mathsf{T}(\mathfrak{f})^\theta = T^\theta_{0:0+}$.
Let $\mathfrak{f}'$ be the splitting field of $\mathsf{T}$.
We will use the assumption that $a$ is $\theta$-fixed to show that $\bar\mu_\theta |N_{\mathfrak{f}'/\mathfrak{f}}(\check a (\mathfrak{f}'^\times))=1$.  Then we will show that the latter fact contradicts the regularity of $\mu$.
(This is the same strategy used in 
\cite[Lemma 10.4]{MR1106911} and \cite[Lemma 8.1]{MR3027804}.)

We observe now that if $u\in \mathfrak{f}'^\times$ then 
$$\theta ( N_{\mathfrak{f}'/\mathfrak{f}}(\check a (u)))
= N_{\mathfrak{f}'/\mathfrak{f}}(\theta (\check a (u)))
= N_{\mathfrak{f}'/\mathfrak{f}}(\check a (u)).$$
Therefore, $$  N_{\mathfrak{f}'/\mathfrak{f}}(\check a (\mathfrak{f}'^\times)) \subset \mathsf{T}(\mathfrak{f})^\theta.$$
But the image of $\check a$ is a connected subgroup of $\mathsf{T}^\theta$ and hence  $$  N_{\mathfrak{f}'/\mathfrak{f}}(\check a (\mathfrak{f}'^\times)) \subset (\mathsf{T}^\theta)^\circ (\mathfrak{f})\subset   (\mathsf{H}_x^{\circ \theta})^\circ (\mathfrak{f}) .$$
Therefore, Lemma \ref{epconn} implies 
$\varepsilon^+_\theta |  N_{\mathfrak{f}'/\mathfrak{f}}(\check a (\mathfrak{f}'^\times)) =1$.   Proposition 2.3(b) in \cite{MR1106911} implies $\varepsilon^-_\theta |  N_{\mathfrak{f}'/\mathfrak{f}}(\check a (\mathfrak{f}'^\times)) =1$.  It follows that from $\mu\varepsilon_{\theta}|T^{\theta} =1$ that $\bar\mu_\theta  |  N_{\mathfrak{f}'/\mathfrak{f}}(\check a (\mathfrak{f}'^\times)) =1$.

It remains to show that the condition $\bar\mu_\theta  |  N_{\mathfrak{f}'/\mathfrak{f}}(\check a (\mathfrak{f}'^\times)) =1$ contradicts the regularity of $\mu$.  Choose a weak factorization $(\mu_-,\mu_+)$ of $\mu$ in the sense of Definition \ref{weakfact}.  Then $(\bT, \mu_-)$ is a depth zero, tame, elliptic, regular pair relative to $\bH$.  Fact 3.4.11 \cite{KalYu} implies that the character $\bar\mu_-$ of $\mathsf{T}(\mathfrak{f})$ associated to $\mu_-$ is in general position, in the sense of Definition 5.15 \cite{MR0393266}.  But then Corollary 5.18 \cite{MR0393266} implies that $\bar\mu_-$ is nonsingular in the sense that for every root $b$ of $(\mathsf{H}_x^\circ,\mathsf{T})$, the restriction of $\bar\mu_-$ to $N_{\mathfrak{f}''/\mathfrak{f}}(\check b (\mathfrak{f}''^\times))$ is nontrivial, where $\mathfrak{f}''$ is a field over which $\check b$ is split.  In particular, $\bar\mu_- |N_{\mathfrak{f}'/\mathfrak{f}}(\check a (\mathfrak{f}'^\times)) $ must be nontrivial.  Finally, we observe that the elements of $N_{\mathfrak{f}'/\mathfrak{f}}(\check a (\mathfrak{f}'^\times)) $ lie within the image of $[H,H]$ and hence they are annihilated by $\mu_+$.  It follows that $\bar\mu_\theta$ coincides with $\bar\mu_-$ on $N_{\mathfrak{f}'/\mathfrak{f}}(\check a (\mathfrak{f}'^\times)) $.  So regularity implies $\bar\mu_\theta  |  N_{\mathfrak{f}'/\mathfrak{f}}(\check a (\mathfrak{f}'^\times))$ is nontrivial, which is a contradiction.
\end{proof}

\subsubsection{Orbits of involutions}

Let
$$\xi_\mu = \left\{ \theta\in \xi\ :\ \theta (\mathsf{T}) =\mathsf{T},\ \mu\varepsilon_{\theta}|T^\theta =1\right\}.$$
For each $\theta\in \xi$, we have a bijection
\begin{eqnarray*}
\Xi_{\mathsf{T},\mu}/(TH_{x,0})_\theta &\to& \xi_\mu\\
h(TH_{x,0})_\theta &\mapsto& h\cdot \theta .
\end{eqnarray*}
Here, $$(TH_{x,0})_\theta = \left\{ h\in TH_{x,0}\ :\ h\theta (h)^{-1}\in Z\right\},$$ where $\bZ$ is the center of $\bG$.

Suppose $\theta$ is an involution of $\bG$ that stabilizes $\bT$ and fixes $x$. 

\begin{lemma}
$(TH_{x,0+})_\theta = T_\theta H_{x,0+}^\theta$.
\end{lemma}

\begin{proof}
Clearly, $(TH_{x,0+})_\theta \supset T_\theta H_{x,0+}^\theta$.

Now suppose $t\in T$, $h\in H_{x,0+}$ and $th\in (TH_{x,0+})_\theta$.  Then the element $z = th\theta (h)^{-1}\theta (t)^{-1}$ lies in $Z$.  The element $y= h\theta (h)^{-1} = zt^{-1}\theta (t)$ lies $T_{0+}$ and satisfies $\theta (y) = y^{-1}$.  Proposition 2.12 \cite{MR2431732} implies that there exists $u\in T_{0+}$ such that $ y = u\theta (u)^{-1}$.
Let $t' = tu\in T$ and $h' = u^{-1}h\in H_{x,0+}$.

Then $th = t'h'$, where $t\in T_\theta$ and $h'\in H_{x,0+}^\theta$.
This implies $(TH_{x,0+})_\theta \subset T_\theta H_{x,0+}^\theta$.
\end{proof}

The previous lemma implies that the constants
$${\rm m}_T^{TH_{x,0}} (\theta )
= [(TH_{x,0})_\theta : (TH_{x,0})^\theta T_\theta ]$$ and
$${\rm m}_{TH_{x,0+}}^{TH_{x,0}} (\theta )
= [(TH_{x,0})_\theta : (TH_{x,0})^\theta (TH_{x,0+})_\theta ]$$ are identical.

\begin{proposition}\label{maincoeff}
$$\left\langle \xi ,\rho_T^\mu\right\rangle^{\varepsilon^+}_{TH_{x,0}} 
= 
\sum_{\zeta\in T\bs \xi} {\rm m}_T^{TH_{x,0}} (\zeta) 
\left\langle \zeta , \mu\right\rangle^{\varepsilon}_{T},$$
where
\begin{eqnarray*}
\left\langle \zeta ,\mu\right\rangle^{\varepsilon}_{T} 
&=&\dim {\rm Hom}_{T^\theta}(\mu, \varepsilon_\theta)\\
&=&\begin{cases}1,&\text{if }
\mu\varepsilon_\theta |T^\theta =1,\\
0,&\text{otherwise,}
\end{cases}
\end{eqnarray*}
for some (hence all)
$\theta\in \zeta$.
\end{proposition}

\begin{proof}
Let $$R =  \left\{ h\cdot \theta\in \xi \ : h\in \Xi_{\mathsf{T},\mu} \right\}.$$
These are the relevant involutions in $\xi$ since these involutions make a nonzero contribution to $\left\langle \xi ,\rho_T^\mu\right\rangle^{\varepsilon^+}_{TH_{x,0}} $.
The exact contribution a given involution $h\cdot \theta\in R$ makes is the number of double cosets in $(TH_{x,0+})\bs \Xi_{\mathsf{T},\mu}/  (TH_{x,0})^\theta $ associated to $h\cdot \theta$.

We have a bijection
\begin{eqnarray*}
\Xi_{\mathsf{T},\mu}/(TH_{x,0})_\theta &\to& R\\
h(TH_{x,0})_\theta &\mapsto & h\cdot \theta 
\end{eqnarray*}
that yields a bijection
$$(TH_{x,0+})\bs \Xi_{\mathsf{T},\mu}/(TH_{x,0})_\theta \to T\bs R,$$ where $T\bs R$ denotes the set of $T$-orbits of involutions in $R$.

Suppose $h\in \Xi_{\mathsf{T},\mu}$.  As $k$ varies over $(TH_{x,0})_\theta$,  the double cosets
$$(TH_{x,0+}) hk(TH_{x,0})^\theta$$ are precisely the double cosets associated to the involution $h\cdot \theta$.
We need to determine the number $N_h$ of such double cosets.

We observe that the condition
$$(TH_{x,0+}) hk_1(TH_{x,0})^\theta =(TH_{x,0+}) hk_2(TH_{x,0})^\theta$$ is equivalent to
$$(h^{-1}TH_{x,0+}h)_\theta k_1(TH_{x,0})^\theta =(h^{-1}TH_{x,0+}h)_\theta k_2(TH_{x,0})^\theta . $$
Next, we observe that $g\mapsto g\theta (g)^{-1}$ identifies
$(TH_{x,0})_\theta / (TH_{x,0})^\theta$ a subgroup of $Z$.  Hence $\mathscr{A} = (TH_{x,0})_\theta / (TH_{x,0})^\theta$ is an abelian group.   Let $\mathscr{B}$ be the image of $(h^{-1}TH_{x,0+}h)_\theta$ in $\mathscr{A}$.   Then
$$N_h = [\mathscr{A}:\mathscr{B}] =
[h\mathscr{A} h^{-1} : h\mathscr{B}h^{-1} ] = {\rm m}_{TH_{x,0+}}^{TH_{x,0}}(h\cdot \theta) ={\rm m}_T^{TH_{x,0}} (h\cdot\theta ).$$
Our assertion follows.
\end{proof}

\subsection{The main theorem}

In the introduction, we stated our main result:

\begin{theorem}\label{mainthm}
If $(\bT,\mu)$ is a tame, elliptic, regular pair  relative to the connected, reductive $F$-group $\bG$ and if $\Theta$ is a $G$-orbit of involutions of $G$ then $$\langle \Theta , \pi (\mu)\rangle_G^1 =\sum_{\zeta\in T\bs \Theta_{\bT}} {\rm m}_T^G(\zeta)\ \langle \zeta ,\mu \rangle_T^\varepsilon.$$
\end{theorem}

\begin{proof} At this point,  we have developed all of the details needed for the proof.  It only remains to recapitulate what we have done and fit the various pieces into the strategy outlined in the introduction.

The first step  is to state the main result of \cite{ANewHM} using the notations from \S\ref{sec:indweights} and \S \ref{sec:compatsec} in this paper:
\begin{equation}\langle \Theta,\pi(\mu)\rangle_G^1 = \sum_{\vartheta\in H_x\bs \Theta_{H_x}} {\rm m}^G_{H_x}(\vartheta)\, \langle \vartheta , \rho\rangle_{H_x}^{\varepsilon^+}.\end{equation}
The (permissible) representation $\rho$ is the induced representation 
$$\rho = {\rm ind}_{TH_{x,0}}^{H_x} (\rho_T^\mu),$$
where $\rho_T^\mu$ is defined in \cite[\S2.4--2.6]{ANewTasho}.
According to the theory in \S\ref{sec:standardMackey}, we have
\begin{equation}\langle \vartheta , \rho\rangle_{H_x}^{\varepsilon^+} = \sum_{\xi \in (TH_{x,0})\bs \vartheta} {\rm m}_{TH_{x,0}}^{H_x}(\xi)\ \langle \xi ,\rho_T^\mu\rangle^{\varepsilon^+}_{TH_{x,0}} .\end{equation} 
(For simplicity, the theory in \S\ref{sec:standardMackey} was developed in the context of finite groups, but Mackey's formula
also holds for smooth representations of $p$-adic groups and the ideas in \S\ref{sec:standardMackey} can easily be adapted to this setting.)

Next, we use Lemma \ref{transofind} and Equations (1) and (2) to deduce:
\begin{equation}\langle \Theta,\pi(\mu)\rangle_G^1 =  \sum_{\xi \in (TH_{x,0})\bs \Theta_{TH_{x,0}}} {\rm m}_{TH_{x,0}}^{G}(\xi)\ \langle \xi ,\rho_T^\mu\rangle^{\varepsilon^+}_{TH_{x,0}}.\end{equation}
Then  we apply Proposition \ref{maincoeff}, which says
\begin{equation}\left\langle \xi ,\rho_T^\mu\right\rangle^{\varepsilon^+}_{TH_{x,0}} 
= 
\sum_{\zeta\in T\bs \xi_{\bT}} {\rm m}_T^{TH_{x,0}} (\zeta) 
\left\langle \zeta , \mu\right\rangle^{\varepsilon}_{T}.\end{equation}
 Combining Equations (3) and (4) and applying Lemma \ref{transofind} now finishes the proof.\end{proof}

\bibliographystyle{amsalpha}
\bibliography{JLHrefs}

\end{document}